\documentclass[12pt]{article}
\usepackage[titletoc,title]{appendix}
\usepackage{amsmath,amsfonts,amssymb,bm,amsthm}
\usepackage[colorlinks,citecolor=blue,linkcolor=blue]{hyperref}
\numberwithin{equation}{section}
\newtheorem{theorem}{Theorem}[section]
\newtheorem{lemma}[theorem]{Lemma}

\newtheorem{proposition}[theorem]{Proposition}
\theoremstyle{definition}

\newtheorem{definition}[theorem]{Definition}
\newtheorem{assumption}[theorem]{Assumption}
\theoremstyle{remark}
\newtheorem{remark}[theorem]{Remark}

\usepackage[T1]{fontenc}
\usepackage[top=1in, bottom=1in, left=1in, right=1in]{geometry}
\usepackage{fancyhdr}
\usepackage{enumerate}
\usepackage{xcolor,cancel}
\usepackage{relsize}

\usepackage{cancel}

\allowdisplaybreaks

\begin{document}

\title{Invariant subspaces of elliptic systems I:\\
pseudodifferential projections}
\author{Matteo Capoferri\thanks{MC:
School of Mathematics,
Cardiff University,
Senghennydd Rd,
Cardiff CF24~4AG,
UK;\newline
CapoferriM@cardiff.ac.uk,
\url{http://mcapoferri.com}.
}
\and
Dmitri Vassiliev\thanks{DV:
Department of Mathematics,
University College London,
Gower Street,
London WC1E~6BT,
UK;
D.Vassiliev@ucl.ac.uk,
\url{http://www.ucl.ac.uk/\~ucahdva/}.
}}

\renewcommand\footnotemark{}

\date{8 February 2022}

\maketitle
\begin{abstract}
Consider an elliptic self-adjoint pseudodifferential operator $A$ acting on $m$-columns of half-densities on a closed manifold $M$, whose principal symbol is assumed to have simple eigen\-values. We show existence and uniqueness of $m$ orthonormal pseudo\-differential projections commuting with the operator $A$ and provide an algorithm for the computation of their full symbols, as well as explicit closed formulae for their subprincipal symbols. Pseudodifferential projections yield a decomposition of $L^2(M)$ into invariant subspaces under the action of $A$ modulo $C^\infty(M)$. Furthermore, they allow us to decompose $A$ into $m$ distinct sign definite pseudodifferential operators. Finally, we represent the modulus and the Heaviside function of the operator $A$ in terms of pseudodifferential projections and discuss physically meaningful examples.

\


\

{\bf 2020 MSC classes: }
primary 
58J40;
secondary
47A15,
35J46,
35J47,
35J48,~58J05.

\end{abstract}

\tableofcontents

\allowdisplaybreaks

\section{Statement of the problem}
\label{Statement of the problem}

Let $M$ be a closed connected manifold of dimension $d\ge2$.
We denote local coordinates on~$M$ by $x=(x^1,\dots,x^d)$.

Let
$C^\infty(M)$ be the vector space of
$m$-columns of smooth complex-valued half-densities over $M$ equipped with inner product
\begin{equation}
\label{inner product on m-columns of half-densities}
\langle v,w\rangle:=\int_M v^*w\,dx\,,
\end{equation}
where $dx:=dx^1\dots dx^d$.
Here and further on the star stands for Hermitian conjugation when applied to matrices
and for adjunction with respect to \eqref{inner product on m-columns of half-densities}
when applied to operators.
By $L^2(M)$ we denote the closure of $C^\infty(M)$ with respect to
\eqref{inner product on m-columns of half-densities}.
Of course, our function spaces depend on the choice of the natural number $m$ but in order to simplify
notation we suppress this dependence.
Thus, throughout this paper $m\ge2$ is a fixed natural number.

By $H^s(M)$ we denote the usual Sobolev space, i.e.~the vector space of $m$-columns of half-densities
that are square integrable together with their partial derivatives up to order~$s$.
By $\Psi^s$ we denote the space of classical pseudodifferential operators of order $s$
with polyhomogeneous symbols, acting from $H^s(M)$ to $L^2(M)$.
For an operator $P\in\Psi^s$ we denote its matrix-valued principal and subprincipal symbols by
$P_\mathrm{prin}$ and $P_\mathrm{sub}$ respectively.
Of course, these are scalar matrix-functions in
$C^\infty(T^*M\setminus\{0\};\operatorname{Mat}(m,\mathbb{C}))$
of degree of homogeneity in momentum $s$ and $s-1$.
We also introduce refined notation for the principal symbol. Namely, we denote by $(\,\cdot\,)_{\mathrm{prin},s}$ the principal symbol of the expression within brackets, regarded as an operator in $\Psi^{-s}$. To appreciate the need for such notation, consider the following example. Let $B$ and $C$ be pseudodifferential operators in $\Psi^{-s}$ with the same principal symbol. Then, as an operator in $\Psi^{-s}$, $B-C$ has vanishing principal symbol: $(B-C)_{\mathrm{prin},s}=0$. But this tells us that $B-C$ is, effectively, an operator in $\Psi^{-s-1}$ and, as such, it may have nonvanishing principal symbol $(B-C)_{\mathrm{prin},s+1}$. This refined notation will be used whenever there is risk of confusion.

\begin{definition}
\label{definition of pseudodifferential projection}
We say that $P\in\Psi^0$ is an \emph{orthogonal pseudodifferential projection} if
\begin{equation}
\label{definition of pseudodifferential projection condition 2}
P^2=P\mod\Psi^{-\infty},
\end{equation}
\begin{equation}
\label{definition of pseudodifferential projection condition 3}
P^*=P\mod\Psi^{-\infty}.
\end{equation}
\end{definition}

\begin{definition}
\label{definition of system of pseudodifferential projections}
We call a set of $m$ orthogonal pseudodifferential projections $\{P_j\}$
an \emph{ortho\-normal pseudodifferential basis} if their principal symbols
are rank 1 matrix-functions and
\begin{equation}
\label{definition of system of pseudodifferential projections condition 1}
P_jP_k=0\mod\Psi^{-\infty}\quad\forall j\ne k,
\end{equation}
\begin{equation}
\label{definition of system of pseudodifferential projections condition 2}
\sum_jP_j=\operatorname{Id}\mod\Psi^{-\infty},
\end{equation}
where $\operatorname{Id}\in\Psi^0$ is the identity operator.
\end{definition}

It is natural to ask the following questions.

\

\textbf{Question 1\ }
Does there exist a nontrivial operator $P$ satisfying Definition~\ref{definition of pseudodifferential projection}?

\

\textbf{Question 2\ }
Assuming that the answer to Question 1 is positive,
can we choose the $P_j$'s so that they satisfy Definition~\ref{definition of system of pseudodifferential projections}?

\

The issue here is that in order to construct these pesudodifferential projections one 
has to determine the lower order (of degree of homogeneity $-1,-2,\,\dots$) components
of the symbols of the $P_j$'s so as to satisfy
\eqref{definition of pseudodifferential projection condition 2}--\eqref{definition of system of pseudodifferential projections condition 2}.
This requires solving an infinite sequence of heavily overdetermined systems of algebraic equations,
and it is not \emph{a priori} clear that these systems have solutions.
We would like to point out that great care is needed
in performing this analysis because our operators have matrix-valued symbols which in general
do not commute.

Dealing with projections in infinite-dimensional spaces is known to be a challenging task
and we believe that addressing Questions 1 and 2
is of interest in its own right.
However, these pseudodifferential projections reveal their true potential when applied to
the study of elliptic and hyperbolic systems of partial differential equations.

Let $A\in\Psi^s$, $s\in\mathbb{R}$, $s>0$, be an elliptic self-adjoint linear pseudodifferential operator,
where ellipticity means that
\begin{equation*}
\label{definition of ellipticity}
\det A_\mathrm{prin}(x,\xi)\ne0,\qquad\forall(x,\xi)\in T^*M\setminus\{0\}.
\end{equation*}

We impose the following  crucial assumption.

\begin{assumption}
\label{assumption about eigenvalues of principal symbol being simple}
The matrix-function $A_\mathrm{prin}(x,\xi)$ has simple eigenvalues.
\end{assumption}

We denote by $m^+$ (resp.~$m^-$) the number of positive (resp.~negative) eigenvalues of
$A_\mathrm{prin}(x,\xi)$.
We denote by $h^{(j)}(x,\xi)$
the eigenvalues of $A_\mathrm{prin}(x,\xi)$ enumerated in increasing order, with positive index $j=1,2,\ldots, m^+$ for positive $h^{(j)}(x,\xi)$ and negative index $j=-1,-2,\ldots, -m^-$ for negative $h^{(j)}(x,\xi)$.
Clearly, self-adjointness, ellipticity and connectedness of $M$ imply that
\begin{enumerate}[(i)]
\itemsep.1em
\item
the $h^{(j)}$ are scalar nonvanishing smooth real-valued functions on $T^*M\setminus\{0\}$,
\item
$m^+$ and $m^-$ are constant and
\item
$m^++m^-=m$.
\end{enumerate}

The spectrum of our operator $A:H^s(M)\to L^2(M)$ is discrete and accumulates to infinity.
More precisely,
if $m^+\ge1$ the spectrum accumulates to $+\infty$,
if $m^-\ge1$ the spectrum accumulates to $-\infty$, 
and if $m^+\ge1$ and $m^-\ge1$ the spectrum accumulates to~$\pm\infty$.

By $P^{(j)}(x,\xi)$ we denote the eigenprojection of $A_\mathrm{prin}(x,\xi)$  corresponding to the eigenvalue
$h^{(j)}(x,\xi)$. Assumption~\ref{assumption about eigenvalues of principal symbol being simple}
tells us that the matrix-functions $P^{(j)}(x,\xi)$ are rank 1. These eigenprojections satisfy
\begin{equation}
\label{principal symbol in terms of eigenvalues and eigenprojections}
A_\mathrm{prin}=\sum_j h^{(j)}P^{(j)}
\end{equation}
and
\begin{equation}
\label{principal symbols sum to identity}
\sum_j P^{(j)}=I,
\end{equation}
where $I$ is the $m\times m$ identity matrix.

\

\textbf{Question 3\ }
Assuming that the answer to Question 2 is positive,
can we choose the $P_j$'s so that
they commute with the operator $A$
\begin{equation}
\label{commutation condition}
[A,P_j]=0\mod\Psi^{-\infty}
\end{equation}
and
\begin{equation}
\label{definition of pseudodifferential projection condition 1bis}
(P_j)_\mathrm{prin}=P^{(j)}?
\end{equation}

\

Here, in addition to the issues highlighted in relation to Questions 1 and 2, the extra difficulty is
that the infinite sequence of overdetermined systems involves the $h^{(j)}$'s as well
as the lower order components of the full symbol of the operator $A$.

Question 3 is important in that an affirmative answer would yield a collection of projections
compatible with $A$, leading to two further natural questions.

\

\textbf{Question 4\ }
Can we exploit the pseudodifferential projections $P_j$ to advance the current understanding of
spectral asymptotics for elliptic systems?

\

\textbf{Question 5\ }
Can we exploit the pseudodifferential projections $P_j$ to advance the current understanding of
propagation of singularities for hyperbolic systems?

\

The goal of this paper is develop a comprehensive theory of pseudodifferential projections
so as to positively answer Questions 1, 2 and 3, building upon earlier results.
Questions 4 and 5 are addressed in the companion paper \cite{part2}.

\section{Main results}
\label{Main results}

Pseudodifferential projections have been used over the years, under different names and with varying degree of awareness, in many areas of mathematical analysis. 

Mathematicians working with Toeplitz operators found in `pseudodifferential subspaces' a useful tool \cite{baum, gohberg}, and Birman and Solomyak \cite{birman} studied and characterised subspaces of sections of vector bundles which can be the range of a pseudodifferential projection.  Birman and Solomyak also provided an abstract formula for a single pseudodifferential projection with given principal symbol, obtained by integrating an appropriate resolvent along a carefully chosen path in the complex plane, see \cite[Lemma~3]{birman}.

Amongst applications in topology and index theory it is worth mentioning the works of Wojciechowski \cite{woj1, woj2}, who analysed the topological properties of the space of equivalence classes of projections with the same principal symbol (i.e.~modulo a compact operator) by means of Fredholm theory. 

A flourishing avenue of research involving pseudodifferential projections in various forms is the study of boundary value problems for elliptic operators. 
A distinguished example is the celebrated Calder\'on projector \cite{calderon,seeley1, seeley2, grieser}.
It is known
\cite[Vol.~III]{Hor}\cite{sternin}
that for an arbitrary elliptic operator on a manifold with boundary one cannot, in general, impose boundary conditions satisfying the Shapiro--Lopatinski condition. This leads to ill-defined (non Fredholm) boundary value problems. The use of pseudodifferential projections proved useful in attempts to generalise elliptic theory to such operators so as to obtain Fredholm boundary value problems, see, e.g., \cite{woj3, woj4, woj5, sternin}.

A considerable advancement in the understanding of pseudodifferential projections is due to Bolte and Glaser \cite{bolte}, who,  relying on a strategy by Cordes \cite{cordes3},  in the setting of semiclassical analysis,  construct pseudo\-differential projections in the spirit of Riesz projections. Their results establish existence and identify the minimal set of conditions that guarantee uniqueness. There are a number of differences between our approach and that of Bolte and Glaser.
\begin{enumerate}[(i)]
\itemsep.1em
\item They work with semiclassical operators on $\mathbb{R}^d$ as opposed to classical ones on a manifold~$M$.

\item The (local) construction of the symbol using Riesz projections requires one to compute the symbol of the resolvent $(A-\lambda\mathrm{Id})^{-1}\mod \Psi^{-\infty}$, which, in turn, necessitates parameter-dependent pseudodifferential calculus.  Our algorithm at every step uses and produces invariant objects and it does not involve parameter-dependent symbol classes or partitions of unity.


\item Whereas the Riesz projection method is essentially linked to an operator $A$, our approach regards pseudodifferential projections as abstract objects, which exist and can be constructed independently of $A$.  This has the advantage of shedding light on the structure of their symbols, as well as clarifying how the degrees of freedom are used up when imposing the defining conditions
\eqref{definition of pseudodifferential projection condition 2}--\eqref{definition of pseudodifferential projection condition 1bis}.
\end{enumerate}


For specific operators, such as the Dirac operator or matricial versions of the Klein--Gordon operator,  alternative approaches to the problem have been proposed, ones that avoid dealing with pseudodifferential projections by constructing almost-unitary operators that `microlocally diagonalise' the matrix operator at hand, see, for example, \cite{cordes,LF91,BR99,NS04,PST03}.  Pseudodifferential projections are, in a sense, `more fundamental' objects than almost-unitary operators: we refer the reader to our companion paper \cite{part2} for further comments on this matter.

We should also mention that pseudodifferential projections appeared in the form of (approximate) spectral projections in publications on the spectral theory of elliptic systems, though some of these publications are known to contain mistakes, see~\cite[Sec.~11]{CDV}.

\

Our main results can be summarised in the form of six theorems stated in this section.

\begin{theorem}
\label{Main results theorem 0}
Given a family of $m$ orthonormal rank 1 projections 
\[
P^{(j)}\in C^\infty(T^*M\setminus\{0\};\operatorname{Mat}(m;\mathbb{C}))
\]
positively homogeneous in momentum of degree zero,
there exists an orthonormal pseudo\-differential basis $\{P_j\}\subset \Psi^0$ as per Definition~\ref{definition of system of pseudodifferential projections} with $(P_j)_\mathrm{prin}=P^{(j)}$.
\end{theorem}

One can show\footnote{We are grateful to Daniel Grieser for raising this issue.} that the operators $P_j$ in Theorem~\ref{Main results theorem 0} can be modified, by adding $\Psi^{-\infty}$ terms, in such a way that conditions \eqref{definition of pseudodifferential projection condition 2}--\eqref{definition of system of pseudodifferential projections condition 1} are satisfied exactly, and not merely modulo $\Psi^{-\infty}$. Indeed, consider the zero order pseudodifferential operator $B:=\frac12\sum_jb_j(P_j+P_j^*)\,$, where the $b_j$ are some distinct real numbers. Its essential spectrum is the set of $m$ points~$b_j\,$. For a given $j\,$, choose a contour $C_j$ in the complex plane which encircles the point $b_j$ and no other points of the essential spectrum and avoids isolated eigenvalues of finite multiplicity. Then integration of $(B-\lambda\operatorname{Id})^{-1}$ over the contour $C_j$ will produce the modified pseudo\-differential projection $P_j\,$. Furthermore, if we choose our contours $C_j$ in such a way that they do not intersect and, in total, encircle the whole spectrum of $B$, then condition \eqref{definition of system of pseudodifferential projections condition 2} will also be satisfied exactly.

\begin{theorem}
\label{Main results theorem 1}
Under Assumption~\ref{assumption about eigenvalues of principal symbol being simple},  there exist $m$ pseudodifferential operators $P_j\in\Psi^0$ satisfying
Definition~\ref{definition of system of pseudodifferential projections}
and conditions
\eqref{commutation condition}, \eqref{definition of pseudodifferential projection condition 1bis},
and these are uniquely determined, modulo $\Psi^{-\infty}$, by the operator~$A$.
\end{theorem}

Of course, Theorem~\ref{Main results theorem 0}
follows from Theorem~\ref{Main results theorem 1}, but we listed them as separate results
for the sake of logical clarity.
Unlike Theorem~\ref{Main results theorem 0}, we do not believe that for a general operator $A$ it is possible to adjust the choice of our pseudodifferential projections $P_j$ in Theorem~\ref{Main results theorem 1} so as to satisfy the commutation conditions \eqref{commutation condition} exactly whilst maintaining exact conditions \eqref{definition of pseudodifferential projection condition 2}--\eqref{definition of system of pseudodifferential projections condition 2}.

Note that Theorem~\ref{Main results theorem 1} cannot be obtained by elementary function-analytic
arguments involving an expansion over eigenvalues and eigenfunctions of the operator $A$.
Theorem~\ref{Main results theorem 1} is to do with the structure of the principal symbol of the operator $A$,
an object which is not detected by the Spectral Theorem.  A semiclassical version of Theorem~\ref{Main results theorem 1} was obtained, with the \emph{caveat}s discussed above,  in \cite{bolte}.

The $m$ orthogonal projections $P_j$ from Theorem \ref{Main results theorem 1} effectively decompose $L^2(M)$
into $m$~infinite-dimensional subspaces
which are invariant under the action of the operator $A$.
Of course, since the construction of our projections is approximate, modulo $\Psi^{-\infty}$,
the resulting decomposition of $L^2(M)$ is also approximate, modulo $C^\infty(M)$:
\begin{equation*}
\label{decomposition of L2}
AP_jL^2(M)\subseteq P_jL^2(M)\mod C^\infty(M).
\end{equation*}

Remarkably, Theorem \ref{Main results theorem 1} will be established
by devising an explicit algorithm leading to the determination of the full symbols of the
pseudodifferential projections $P_j$'s, see sub\-sections~\ref{The algorithm}
or~\ref{The simplified algorithm}.
In particular, we will obtain the following result.

\begin{theorem}
\label{Main results theorem 2}
The explicit formula for the subprincipal symbol of the pseudodifferential projection $P_j$ reads
\begin{equation}
\label{Main results theorem 2 equation 1}
\begin{split}
(P_j)_\mathrm{sub}
&
=
\frac i2 \{P^{(j)},P^{(j)}\}-i \,P^{(j)}\{P^{(j)},P^{(j)}\}P^{(j)}
\\
&
+
\sum_{l\ne j}
\frac
{
P^{(j)}(A_\mathrm{sub}-iQ^{(j)})P^{(l)}
+
P^{(l)}(A_\mathrm{sub}+iQ^{(j)})P^{(j)}
}
{h^{(j)}-h^{(l)}}\,,
\end{split}
\end{equation}
where
\begin{equation}
\label{Main results theorem 2 equation 2}
Q^{(j)}
:=
\frac12
\bigl(
\{A_\mathrm{prin},P^{(j)}\}
-
\{P^{(j)},A_\mathrm{prin}\}
\bigr).
\end{equation}
\end{theorem}

In formulae
\eqref{Main results theorem 2 equation 1}
and
\eqref{Main results theorem 2 equation 2}
curly brackets denote the Poisson bracket
\begin{equation}
\label{poisson brackets}
\{B,C\}:=\sum_{\alpha=1}^d(B_{x^\alpha} C_{\xi_\alpha}- B_{\xi_\alpha} C_{x^\alpha})
\end{equation}
on matrix-functions on the cotangent bundle.
Further on in the paper we will also make use
of the generalised Poisson bracket
\begin{equation*}
\label{generalised poisson brackets}
\{ B,C,D\}:=\sum_{\alpha=1}^d(B_{x^\alpha} C D_{\xi_\alpha}- B_{\xi_\alpha} C D_{x^\alpha}).
\end{equation*}
Let us emphasise that the order of terms in matrix-valued Poisson brackets matters;
for example, the usual properties $\{f,f\}=0$ and
$\{f,g\}=-\{g,f\}$
from Hamiltonian mechanics no longer hold
if the scalar functions $f$ and $g$ are replaced by matrix-functions.

Let us point out that having an explicit formula for the subprincipal symbol of projections $P_j$
is important for applications. For example, the matrix trace of $(P_j)_\mathrm{sub}$ appears in the second
Weyl coefficient of the eigenvalue counting function(s) of the operator $A$,
see \cite{CDV}. Failure to appreciate this fact led to a number of incorrect publications.
For a long time it was assumed that Safarov \cite{safarov}
did obtain the formula for the second
Weyl coefficient, fixing previous mistakes, but his formula also turned out to be wrong.
With the benefit of hindsight, Safarov's mistake can be traced back to
the incorrect assumption that $(P_j)_\mathrm{sub}=0$.
A brief account of the troubled history of the subject is given in \cite[Section~11]{CDV}.

\begin{definition}
\label{definition of sign definiteness modulo}
We say that a symmetric pseudodifferential operator $B$ is \emph{nonnegative (resp.~nonpositive) modulo} $\Psi^{-\infty}$
and write
\[
B\ge0\mod\Psi^{-\infty}\qquad(\text{resp.}\ B\le0\mod\Psi^{-\infty})
\]
if there exists a symmetric operator $C\in\Psi^{-\infty}$
such that $B+C\ge0$ (resp.~$B+C\le0$).
\end{definition}

\begin{theorem}
\label{Main results theorem 3}
We have
\begin{align}
\label{Main results theorem 3 equation 1}
&P_j^*AP_j\ge0\mod\Psi^{-\infty}\quad\text{for}\quad j=1,\dots,m^+,
\\
\label{Main results theorem 3 equation 2}
&P_j^*AP_j\le0\mod\Psi^{-\infty}\quad\text{for}\quad j=-1,\dots,-m^-.
\end{align}
\end{theorem}

Note that the operators $P_j^*AP_j$ appearing in
Theorem~\ref{Main results theorem 3} are not elliptic,
\begin{equation*}
\label{Main results theorem 3 equation not elliptic}
\det(P_j^*AP_j)_\mathrm{prin}(x,\xi)=0\qquad\forall(x,\xi)\in T^*M\setminus\{0\},
\end{equation*}
therefore, proving that they are sign semidefinite modulo $\Psi^{-\infty}$ is a delicate matter.
The fact that their principal symbols are sign semidefinite does not, on its own,
imply that the operators are sign semidefinite --- it does not even imply that they are semibounded.

Let $\lambda_k$ be the eigenvalues of the operator $A$ enumerated with account of multiplicity
and $v_k$ be the corresponding orthonormal eigenfunctions. The choice of a particular enumeration is
irrelevant for our purposes.

Consider the operator \emph{modulus} of $A$ defined in accordance with
\begin{equation}
\label{definition of the modulus of the operator A}
|A|:=\sum_{k}|\lambda_k|\,\langle v_k,\,\cdot\,\rangle\,v_k\,.
\end{equation}

\begin{theorem}
\label{Main results theorem 4}
The operator $|A|$ is pseudodifferential and
\begin{equation}
\label{Main results theorem 4 equation 1}
|A|=\sum_{j=1}^{m^+}AP_j-\sum_{j=1}^{m^-}AP_{-j}\mod\Psi^{-\infty}.
\end{equation}
Furthermore, the explicit formula for
the subprincipal symbol of the operator $|A|$ reads
\begin{multline}
\label{Main results theorem 4 equation 2}
|A|_\mathrm{sub}
=
\sum_{j,k}
\frac{h^{(j)}+h^{(k)}}{|h^{(j)}|+|h^{(k)}|}
P^{(j)}A_\mathrm{sub}P^{(k)}
\\
+
\frac i2
\sum_{j,k}
\frac1{|h^{(j)}|+|h^{(k)}|}
P^{(j)}
(
\{A_\mathrm{prin},A_\mathrm{prin}\}
-
\{|A|_\mathrm{prin},|A|_\mathrm{prin}\}
)
P^{(k)}.
\end{multline}
\end{theorem}

Let $\theta:\mathbb{R}\to\mathbb{R}$,
\begin{equation*}
\label{Heaviside function}
\theta(z):=
\begin{cases}
0&\text{if}\quad z\le0,
\\
1&\text{if}\quad z>0
\end{cases}
\end{equation*}
be the Heaviside function.
Consider the operator
\begin{equation}
\label{definition of the Heaviside function of the operator A}
\theta(A):=\sum_{k:\ \lambda_k>0}\langle v_k,\,\cdot\,\rangle\,v_k\,.
\end{equation}

\begin{theorem}
\label{Main results theorem 5}
The operator $\theta(A)$ is pseudodifferential and
\begin{equation}
\label{Main results theorem 5 equation 1}
\theta(A)=\sum_{j=1}^{m^+}P_j\mod\Psi^{-\infty}.
\end{equation}
\end{theorem}

Of course, Theorem~\ref{Main results theorem 2}
immediately gives us an explicit formula for $[\theta(A)]_\mathrm{sub}$.

\begin{remark}
We should point out that the setting of our paper is not the most general setting in which one can construct pseudodifferential projections satisfying conditions \eqref{definition of pseudodifferential projection condition 2}--\eqref{definition of system of pseudodifferential projections condition 2},  \eqref{commutation condition}--\eqref{definition of pseudodifferential projection condition 1bis}.  Note, for example, that the ellipticity of $A$ or the fact that $A$ is of positive order are not really needed in the proof of Theorem~\ref{Main results theorem 1} or in the construction algorithm leading up to Theorem~\ref{Main results theorem 2} (but they are needed in Theorem~\ref{Main results theorem 3}).  Furthermore,  our algorithm can be extended without much effort to operators acting on vector bundles --- and, in particular, to operators acting on differential forms (e.g., the operator curl).  Different techniques may allow one to generalise the results even further. The reason why we refrain from carrying out such generalisations in the current paper is twofold: (i) we aim to write a paper accessible to a wide audience, not necessarily limited to (microlocal) analysts,  and (ii) we are motivated by applications in spectral theory (see Theorems~\ref{Main results theorem 2}, \ref{Main results theorem 3} and \cite{part2,diagonalisation}).
We will address certain generalisations elsewhere.
\end{remark}

The paper is structured as follows.

In Section~\ref{Pseudodifferential projections: general theory} we develop the general theory of pseudodifferential projections: in subsection~\ref{Construction of a single pseudodifferential projection} we construct a single pseudodifferential projection, in subsection~\ref{Construction of a basis of pseudodifferential projections} we construct an orthonormal basis of pseudodifferential projections and in subsection~\ref{Commutation with an elliptic operator} we show that the latter can be chosen in such a way that it commutes with our elliptic operator $A$ and that this determines the projections uniquely. The results of Section~\ref{Pseudodifferential projections: general theory} are summarised in subsection~\ref{The algorithm} in the form of an algorithm for the construction of the full symbol of pseudodifferential projections.

In Section~\ref{Commutation with an elliptic operator: revisited Dima} we show that a set of $m$ pseudodifferential projections commuting with an elliptic operator $A\in \Psi^s$, $s>0$, are automatically orthonormal and sum to the identity operator, modulo $\Psi^{-\infty}$. This leads to a simplified algorithm for the construction of their full symbols, presented in subsection~\ref{The simplified algorithm}.

In Section~\ref{Subprincipal symbol of pseudodifferential projections} we carry out the first step of our algorithm and obtain a closed explicit formula for the subprincipal symbol of pseudodifferential projections.

Section~\ref{A positivity result} is concerned with the proof of Theorem~\ref{Main results theorem 3}, which consists in a rigorous formulation of the fact that one can use pseudodifferential projections to construct $m$ distinct sign definite operators (modulo $\Psi^{-\infty}$) out of $A$.

Results from Sections \ref{Commutation with an elliptic operator: revisited Dima} and~\ref{Subprincipal symbol of pseudodifferential projections} are employed in Section~\ref{Modulus and Heaviside function of an elliptic system} to represent modulus and Heaviside function of $A$ in terms of pseudodifferential projections. This yields a simpler -- compared to those available in the literature -- algorithm for the calculation of the full symbols of $|A|$ and $\theta(A)$, as well as explicit formulae for $|A|_\mathrm{sub}$ and $[\theta(A)]_\mathrm{sub}$.

Lastly, in Section~\ref{Applications} we discuss three applications of our results: to the massless Dirac operator on a closed 3-manifold, to the operator of linear elasticity (Lam\'e operator) on a 2-torus and to the Dirichlet-to-Neumann map of linear elasticity in 3D. 

\section{Pseudodifferential projections: general theory}
\label{Pseudodifferential projections: general theory}

The goal of this section is to develop a comprehensive and self-contained theory of pseudo\-differential projections in $L^2(M)$\footnote{Recall that in this paper $L^2(M)$ denotes the space of $m$-columns of square integrable complex-valued half-densities.}, including an explicit construction of their full symbols. This analysis, which we believe to be of interest in its own right, will answer Questions~1, 2 and 3 from Section~\ref{Statement of the problem}, and lay rigorous foundations for the use of pseudodifferential projections in the study of spectral asymptotics of elliptic systems
carried out in our companion paper \cite{part2}.

\subsection{Construction of a single pseudodifferential projection}
\label{Construction of a single pseudodifferential projection}

In this subsection we prove the existence and establish the general structure
of an operator $P_j\in\Psi^0$ satisfying conditions
\eqref{definition of pseudodifferential projection condition 2}--\eqref{definition of pseudodifferential projection condition 3}.
We do this by constructing a sequence
$P_{j,k}\in\Psi^0$, $k=0,1,2,\dots$,
of pseudodifferential operators such that
\begin{equation}
\label{Construction of a single pseudodifferential projection equation 1}
P_{j,k+1}-P_{j,k}\in\Psi^{-k-1},
\end{equation}
\begin{equation}
\label{Construction of a single pseudodifferential projection equation 2}
P_{j,k}^2=P_{j,k}\mod\Psi^{-k-1},
\end{equation}
\begin{equation}
\label{Construction of a single pseudodifferential projection equation 3}
P_{j,k}^*=P_{j,k}\mod\Psi^{-\infty}
\end{equation}
for $k=0,1,2,\dots$.
For $P_{j,0}$ we choose an arbitrary pesudodifferential operator satisfying 
\eqref{Construction of a single pseudodifferential projection equation 2} and \eqref{Construction of a single pseudodifferential projection equation 3}, and construct
subsequent $P_{j,k}$ by solving
\eqref{Construction of a single pseudodifferential projection equation 1}--\eqref{Construction of a single pseudodifferential projection equation 3} recursively.

To this end, we seek $P_{j,k}$, $k=1,2,\dots$, in the form
\begin{equation*}
\label{Construction of a single pseudodifferential projection equation 4}
P_{j,k}=P_{j,k-1}+X_{j,k},
\end{equation*}
where $X_{j,k}\in\Psi^{-k}$ is an unknown pseudodifferential operator
such that
\begin{equation}
\label{Construction of a single pseudodifferential projection equation 4bis}
X_{j,k}=X_{j,k}^* \mod \Psi^{-\infty}.
\end{equation}

Then condition 
\eqref{Construction of a single pseudodifferential projection equation 1}
is automatically satisfied,
whereas solving
\eqref{Construction of a single pseudodifferential projection equation 2}
and
\eqref{Construction of a single pseudodifferential projection equation 3}
reduces to solving
\begin{equation*}
\label{Construction of a single pseudodifferential projection equation 5}
[(P_{j,k-1}+X_{j,k})^2-P_{j,k-1}-X_{j,k}]_\mathrm{prin,k}=0,
\end{equation*}
\begin{equation*}
\label{Construction of a single pseudodifferential projection equation 6}
[(P_{j,k-1}+X_{j,k})^*-P_{j,k-1}-X_{j,k}]_\mathrm{prin,k}=0,
\end{equation*}
which gives us a system of equations for the unknown $(X_{j,k})_\mathrm{prin}\,$.
This system of equations reads
\begin{equation}
\label{Construction of a single pseudodifferential projection equation 7}
P^{(j)}(X_{j,k})_\mathrm{prin}
+
(X_{j,k})_\mathrm{prin}P^{(j)}
-
(X_{j,k})_\mathrm{prin}
=
R_{j,k}\,,
\end{equation}
\begin{equation}
\label{Construction of a single pseudodifferential projection equation 8}
(X_{j,k})_\mathrm{prin}^*
-
(X_{j,k})_\mathrm{prin}
=
0\,,
\end{equation}
where
\begin{equation}
\label{Construction of a single pseudodifferential projection equation 9}
R_{j,k}
:=
-
[(P_{j,k-1})^2-P_{j,k-1}]_\mathrm{prin,k}\,.
\end{equation}
In fact, once one has determined $(X_{j,k})_\mathrm{prin}$ satisfying \eqref{Construction of a single pseudodifferential projection equation 7} and \eqref{Construction of a single pseudodifferential projection equation 8}, it is always possible to choose lower order terms in the symbol of $X_{j,k}$ so as to satisfy \eqref{Construction of a single pseudodifferential projection equation 4bis}.


\begin{lemma}
\label{Construction of a single pseudodifferential projection lemma 1}
The general solution of the system
\eqref{Construction of a single pseudodifferential projection equation 7},
\eqref{Construction of a single pseudodifferential projection equation 8}
reads
\begin{equation}
\label{Construction of a single pseudodifferential projection equation 12}
(X_{j,k})_\mathrm{prin}
=-R_{j,k}+P^{(j)}R_{j,k}+R_{j,k}P^{(j)}+Y_{j,k}+Y_{j,k}^*,
\end{equation}
where $Y_{j,k}$ is an arbitrary matrix-function positively homogeneous in momentum of degree~$-k$ such that
\begin{equation}
\label{Construction of a single pseudodifferential projection equation 13}
Y_{j,k}=P^{(j)}Y_{j,k}(I-P^{(j)}).
\end{equation}
\end{lemma}

\begin{proof}
From the inductive assumption
\[
P_{j,k-1}=P_{j,k-1}^* \mod \Psi^{-\infty}
\]
it follows that $R_{j,k}$ is Hermitian,
\begin{equation}
\label{proof lemma 1 single projection equation 2}
R_{j,k}=R_{j,k}^*.
\end{equation}
Therefore, \eqref{Construction of a single pseudodifferential projection equation 12} satisfies \eqref{Construction of a single pseudodifferential projection equation 8}.

Direct inspection of \eqref{Construction of a single pseudodifferential projection equation 7} tells us that the system \eqref{Construction of a single pseudodifferential projection equation 7}, \eqref{Construction of a single pseudodifferential projection equation 8} has a solution only if
\begin{equation}
\label{proof lemma 1 single projection equation 3}
P^{(j)}
R_{j,k}
(I-P^{(j)})=0.
\end{equation}
Of course, \eqref{proof lemma 1 single projection equation 3} and \eqref{proof lemma 1 single projection equation 2} imply
\begin{equation}
\label{proof lemma 1 single projection equation 4}
(I-P^{(j)})
R_{j,k}
P^{(j)}=0.
\end{equation}

Let us show that \eqref{proof lemma 1 single projection equation 3} is satisfied. We have
\begin{equation*}
\label{proof lemma 1 single projection equation 5}
\begin{split}
P^{(j)}
R_{j,k}
(I-P^{(j)})
&
=
-P^{(j)}
[(P_{j,k-1})^2-P_{j,k-1}]_\mathrm{prin,k}
(I-P^{(j)})
\\
&
=
-
[P_{j,k-1}]_\mathrm{prin,0}[(P_{j,k-1})^2-P_{j,k-1}]_\mathrm{prin,k}[\operatorname{Id}-P_{j,k-1})]_\mathrm{prin,0}
\\
&
=
-
[P_{j,k-1}((P_{j,k-1})^2-P_{j,k-1})(\operatorname{Id}-P_{j,k-1})]_\mathrm{prin,k}
\\
&
=
[((P_{j,k-1})^2-P_{j,k-1})^2]_\mathrm{prin,k}
\\
&
=0\,.
\end{split}
\end{equation*}
In the last step we used the fact that since $(P_{j,k-1})^2-P_{j,k-1}\in \Psi^{-k}$ by inductive assumption, then $((P_{j,k-1})^2-P_{j,k-1})^2\in \Psi^{-2k}$, hence its $k$-principal symbol is zero.

It remains only to substitute \eqref{Construction of a single pseudodifferential projection equation 12} and \eqref{Construction of a single pseudodifferential projection equation 13} into \eqref{Construction of a single pseudodifferential projection equation 7} with account of \eqref{proof lemma 1 single projection equation 3}--\eqref{proof lemma 1 single projection equation 4} and observe that $Y_{j,k}+Y_{j,k}^*$ is the general solution of the homogeneous system
\begin{equation*}
P^{(j)} Y
+
Y P^{(j)}
-
Y=0, \qquad Y=Y^*.
\end{equation*}
\end{proof}

All in all, the above argument establishes the following result.

\begin{theorem}
\label{theorem single pseudodifferential projection}
Given a rank $1$ orthogonal projection $P^{(j)}\in C^\infty(T^*M;\mathrm{Mat}(m,\mathbb{C}))$, the associated orthogonal pseudodifferential projection $P_j\in \Psi^0$ in the sense of Definition~\ref{definition of pseudodifferential projection} exists and is given by
\begin{equation}
\label{theorem single pseudodifferential projection equation 1}
P_j \sim P_{j,0}+\sum_{k=1}^{+\infty} X_{j,k},
\end{equation}
where $P_{j,0}\in \Psi^0$ is an arbitrary operator satisfying $(P_{j,0})_\mathrm{prin}=P^{(j)}$, $P_{j,0}=P_{j,0}^*$, and the operators $X_{j,k}\in \Psi^{-k}$, $k=1,2,\ldots$, are constructed iteratively from $P_{j,0}$ by means of Lemma~\ref{Construction of a single pseudodifferential projection lemma 1}. Here $\sim$ stands for asymptotic expansion in smoothness.
\end{theorem}

Formula \eqref{theorem single pseudodifferential projection equation 1} allows one to explicitly determine the symbol of $P_j$ with arbitrarily high accuracy.

Note that for each $P_j$ at every stage of the iterative process we have $m-1$ complex-valued scalar degrees of freedom, see~\eqref{Construction of a single pseudodifferential projection equation 13}. Because we have $m$ different $P_j$'s, at every step of the iterative process we have a total of $m(m-1)$ complex-valued scalar degrees of freedom.

\subsection{Construction of a basis of pseudodifferential projections}
\label{Construction of a basis of pseudodifferential projections}

In this subsection we establish the existence and the general structure
of an orthonormal pseudodifferential basis, in the sense of Definition~\ref{definition of system of pseudodifferential projections}, thus proving Theorem~\ref{Main results theorem 0}.

Suppose we are given $m$ orthonormal rank $1$ projections $P^{(j)}(x,\xi)$ satisfying \eqref{principal symbols sum to identity}, not necessarily coinciding with the eigenprojections of $A_\mathrm{prin}$.
They determine, via Theorem~\ref{theorem single pseudodifferential projection}, a corresponding family of pseudodifferential projections $P_j$, satisfying
\begin{equation}
\label{conditions principal symbols sum to identity}
\left(\sum_j P_j\right)_\mathrm{prin}= I.
\end{equation}
The task at hand is to exploit the degrees of freedom left in the symbols of the $P_j$'s to satisfy conditions \eqref{definition of system of pseudodifferential projections condition 1} and \eqref{definition of system of pseudodifferential projections condition 2}.

Firstly, with the help of \eqref{principal symbols sum to identity}, let us rewrite formulae
\eqref{Construction of a single pseudodifferential projection equation 12}
and
\eqref{Construction of a single pseudodifferential projection equation 13}
in the equivalent form
\begin{equation}
\label{Construction of a single pseudodifferential projection equation 12a}
(X_{j,k})_\mathrm{prin}
=-R_{j,k}+P^{(j)}R_{j,k}+R_{j,k}P^{(j)}+\sum_{l\ne j}[Y_{j,l,k}+Y_{j,l,k}^*],
\end{equation}
where $Y_{j,l,k}$ is an arbitrary matrix-function positively homogeneous in momentum of degree~$-k$ such that
\begin{equation}
\label{Construction of a single pseudodifferential projection equation 13a}
Y_{j,l,k}=P^{(j)}Y_{j,l,k}P^{(l)}.
\end{equation}

Subsection~\ref{Construction of a basis of pseudodifferential projections} gives us, for each $j$, a sequence of operators $P_{j,k}\in \Psi^0$, $k=1,2,\ldots$, satisfying \eqref{Construction of a single pseudodifferential projection equation 1}--\eqref{Construction of a single pseudodifferential projection equation 3} of the form
\begin{equation*}
P_{j,k}=P_{j,k-1}+X_{j,k},
\end{equation*}
where the principal symbol of $X_{j,k}\in \Psi^{-k}$, $X_{j,k}=X_{j,k}^* \mod \Psi^{-\infty}$, is given by \eqref{Construction of a single pseudodifferential projection equation 12a}.
Satisfying \eqref{definition of system of pseudodifferential projections condition 1} reduces to determining $Y_{j,l,k}$ such that
\begin{equation}
\label{basis of projections equation 1}
P_{n,k}P_{j,k}=0 \mod \Psi^{-k-1} \qquad \forall n \ne j.
\end{equation}

To this end, let $\widetilde X_{j,k}\in\Psi^{-k}$, $\widetilde X_{j,k}=\widetilde X_{j,k}^*$, be such that
\begin{equation}
\label{Construction of a single pseudodifferential projection equation 14}
(\widetilde X_{j,k})_\mathrm{prin}
=-R_{j,k}+P^{(j)}R_{j,k}+R_{j,k}P^{(j)}.
\end{equation}
Then satisfying \eqref{basis of projections equation 1} reduces to solving
\begin{equation}
\label{basis of projections equation 3}
\sum_{l\ne j}
P^{(n)}
(Y_{j,l,k}+Y_{j,l,k}^*)
+
\sum_{l\ne n}
(Y_{n,l,k}+Y_{n,l,k}^*)P^{(j)}
=R_{n,j,k}
\end{equation}
for all $j\ne n$,
where $Y_{j,l,k}$ is of the form 
\eqref{Construction of a single pseudodifferential projection equation 13a}
and
\begin{equation}
\label{Construction of a single pseudodifferential projection equation 16}
R_{n,j,k}
:=
-[(P_{n,k-1}+\widetilde X_{n,k})(P_{j,k-1}+\widetilde X_{j,k})]_\mathrm{prin,k}\,.
\end{equation}
The system \eqref{basis of projections equation 3} amounts to a total of $m(m-1)$ algebraic equations.

\begin{lemma}
\label{basis of projections lemma 1}
The general solution of the system
\eqref{basis of projections equation 3}
reads
\begin{equation}
\label{Construction of a single pseudodifferential projection equation 21}
Y_{j,l,k}
=
\frac12
R_{j,l,k}
+
Z_{j,l,k}\,,
\end{equation}
where $Z_{j,l,k}$ are arbitrary matrix-functions positively homogeneous in momentum of degree~$-k$ such that
\begin{equation}
\label{Construction of a single pseudodifferential projection equation 22}
Z_{j,l,k}=P^{(j)}Z_{j,l,k}P^{(l)},
\end{equation}
\begin{equation}
\label{Construction of a single pseudodifferential projection equation 23}
Z_{j,l,k}^*
=
-
Z_{l,j,k}\,.
\end{equation}
\end{lemma}

\begin{proof}
Formula \eqref{Construction of a single pseudodifferential projection equation 16} implies
\begin{equation*}
\label{proof basis of projections lemma 1 equation 1}
R_{j,l,k}^*=R_{l,j,k}.
\end{equation*}

Direct inspection of \eqref{basis of projections equation 3} tells us that a necessary solvability condition reads
\begin{equation}
\label{proof basis of projections lemma 1 equation 2}
R_{j,l,k}
=
P^{(j)}
R_{j,l,k}
P^{(l)}.
\end{equation}
Let us show that \eqref{proof basis of projections lemma 1 equation 2} is satisfied. We have
\begin{equation*}
\label{proof basis of projections lemma 1 equation 3}
\begin{split}
P^{(j)}
R_{j,l,k}
P^{(l)}
&
=
-
P^{(j)}
[(P_{j,k-1}+\widetilde X_{j,k})(P_{l,k-1}+\widetilde X_{l,k})]_\mathrm{prin,k}
P^{(l)}
\\
&
=
-
[P_{j,k-1}+\widetilde X_{j,k}]_\mathrm{prin,0}
\,
[(P_{j,k-1}+\widetilde X_{j,k})(P_{l,k-1}+\widetilde X_{l,k})]_\mathrm{prin,k}
\,
[P_{l,k-1}+\widetilde X_{l,k}]_\mathrm{prin,0}
\\
&
=
-
[(P_{j,k-1}+\widetilde X_{j,k})^2(P_{l,k-1}+\widetilde X_{l,k})^2]_\mathrm{prin,k}
\\
&
=
-
[(P_{j,k-1}+\widetilde X_{j,k})(P_{l,k-1}+\widetilde X_{l,k})]_\mathrm{prin,k}
\\
&
=
R_{j,l,k}\,.
\end{split}
\end{equation*}
In the above argument we used the fact that
\[
(P_{j,k-1}+\widetilde X_{j,k})^2=P_{j,k-1}+\widetilde X_{j,k}\mod\Psi^{-k-1}
\]
for all $j$.

Of course, \eqref{proof basis of projections lemma 1 equation 2} implies
\begin{equation}
\label{proof basis of projections lemma 1 equation 4}
R_{j,l,k}=P^{(j)}R_{j,l,k}=R_{j,l,k}P^{(l)}.
\end{equation} 
It remains only to substitute \eqref{Construction of a single pseudodifferential projection equation 21} into \eqref{basis of projections equation 3} with account of \eqref{proof basis of projections lemma 1 equation 4} and observe that $Z_{j,l,k}$ is the general solution of the homogeneous system
\begin{equation}
\sum_{l\ne j}
P^{(n)}
(Z_{j,l,k}+Z_{j,l,k}^*)
+
\sum_{l\ne n}
(Z_{n,l,k}+Z_{n,l,k}^*)P^{(j)}=0.
\end{equation}
\end{proof}

As it turns out, condition \eqref{definition of system of pseudodifferential projections condition 2} is automatically satisfied.

\begin{theorem}
\label{theorem identity operator}
Let $\{P_j\}$ be a family of $m$ pseudodifferential operators of order zero with rank 1 principal symbols, satisfying 
\eqref{definition of pseudodifferential projection condition 2}, \eqref{definition of system of pseudodifferential projections condition 1}
and
\begin{equation}
\label{theorem identity operator equation 1}
\sum_j (P_j)_\mathrm{prin}=I.
\end{equation}
Then \eqref{definition of system of pseudodifferential projections condition 2} is also satisfied.
\end{theorem}

\begin{proof}
Let us define
\begin{equation*}
\label{proof theorem identity operator equation 1}
\widetilde{\mathrm{Id}}:=\sum_j P_j
\end{equation*}
and let us put
\begin{equation}
\label{proof theorem identity operator equation 2}
B:=\widetilde{\mathrm{Id}}-\mathrm{Id}.
\end{equation}
The task at hand is to show that $B\in \Psi^{-\infty}$. Arguing by contradiction, suppose there exists a natural number $k$
such that
\begin{equation*}
B\in \Psi^{-k}
\end{equation*}
but 
\begin{equation*}
B \not \in \Psi^{-k-1}.
\end{equation*}

The principal symbol of the operator $B$ is positively homogeneous in momentum of degree $-k$
and has the property
\begin{equation}
\label{10 July 2020 equation 6}
B_\mathrm{prin}(x,\xi)\ne0\quad\text{for some}\quad(x,\xi)\in T^*M\setminus\{0\}.
\end{equation}

On account of 
\eqref{definition of pseudodifferential projection condition 2}
and
\eqref{definition of system of pseudodifferential projections condition 1},
formula 
\eqref{proof theorem identity operator equation 2} 
implies
\begin{equation}
\label{10 July 2020 equation 7}
\widetilde{\operatorname{Id}}\,B\in\Psi^{-\infty}.
\end{equation}
We have
\begin{equation}
\label{10 July 2020 equation 8}
(\widetilde{\operatorname{Id}}\,B)_\mathrm{prin}
=
{\widetilde{\operatorname{Id}}}_\mathrm{prin}B_\mathrm{prin}\,.
\end{equation}
But by \eqref{theorem identity operator equation 1}
\begin{equation}
\label{10 July 2020 equation 9}
{\widetilde{\operatorname{Id}}}_\mathrm{prin}=I,
\end{equation}
so formulae \eqref{10 July 2020 equation 8} and \eqref{10 July 2020 equation 9} imply
\begin{equation}
\label{10 July 2020 equation 10}
(\widetilde{\operatorname{Id}}\,B)_\mathrm{prin}
=
B_\mathrm{prin}\,.
\end{equation}

Formulae \eqref{10 July 2020 equation 6} and \eqref{10 July 2020 equation 10} imply
\begin{equation*}
\label{10 July 2020 equation 11}
(\widetilde{\operatorname{Id}}\,B)_\mathrm{prin,k}(x,\xi)\ne0\quad\text{for some}\quad(x,\xi)\in T^*M\setminus\{0\},
\end{equation*}
which, in turn, implies
\begin{equation*}
\label{10 July 2020 equation 12}
\widetilde{\operatorname{Id}}\,B\not\in\Psi^{-k-1}.
\end{equation*}
The latter contradicts \eqref{10 July 2020 equation 7}.
\end{proof}

All in all, the above arguments establish the following result.

\begin{theorem}
\label{theorem basis of pseudodifferential projections}
Given $m$ orthonormal rank 1 projections 
\[
P^{(j)}\in C^\infty(T^*M;\mathrm{Mat}(m,\mathbb{C})),
\]
there exists an orthonormal pseudodifferential basis $\{P_j\}$ in the sense of Definition~\ref{definition of system of pseudodifferential projections} satisfying $(P_j)_\mathrm{prin}=P^{(j)}$.
Furthermore, we have
\begin{equation}
\label{theorem basis of pseudodifferential projections equaton 1}
P_j \sim P_{j,0}+\sum_{k=1}^{+\infty} X_{j,k},
\end{equation}
where $P_{j,0}\in \Psi^0$ is an arbitrary operator satisfying $(P_{j,0})_\mathrm{prin}=P^{(j)}$, $P_{j,0}=P_{j,0}^*$, and the operators $X_{j,k}\in \Psi^{-k}$, $X_{j,k}=X_{j,k}^*\mod\Psi^{-\infty}$, $k=1,2,\ldots$, are constructed iteratively from $P_{j,0}$ in accordance with \eqref{Construction of a single pseudodifferential projection equation 12a}, \eqref{Construction of a single pseudodifferential projection equation 9},  Lemma~\ref{basis of projections lemma 1} and \eqref{Construction of a single pseudodifferential projection equation 16}.  Here $\sim$ stands for asymptotic expansion in smoothness.
\end{theorem}

Note that examination of formulae
\eqref{Construction of a single pseudodifferential projection equation 22}
and
\eqref{Construction of a single pseudodifferential projection equation 23}
shows that at every stage of the iterative process
we have a total of $\frac{m(m-1)}2$ complex-valued scalar degrees of freedom left in our construction.

\subsection{Commutation with an elliptic operator}
\label{Commutation with an elliptic operator}

In this subsection we will exploit the remaining degrees of freedom left in the symbols of our pseudodifferential basis to impose that individual projections commute with the operator $A$, in accordance with \eqref{commutation condition}. We shall then show that this uniquely determines our pseudodifferential projections modulo $\Psi^{-\infty}$, thus completing the proof of Theorem~\ref{Main results theorem 1}.

Let $A\in \Psi^s$ be as in Section~\ref{Statement of the problem}. Suppose we are given an orthonormal pseudodifferential basis constructed in accordance with subsection~\ref{Construction of a basis of pseudodifferential projections}, whose principal symbols are the eigenprojections of $A_\mathrm{prin}$.

Imposing condition \eqref{commutation condition} is equivalent to requiring, for every $j$,
\begin{equation}
\label{commutation with operator recursive condition}
[A,P_{j,k}]=0 \mod \Psi^{-k-1}
\end{equation}
recursively, for $k=0,1,\ldots$.

For $k=0$, condition \eqref{commutation with operator recursive condition} is automatically satisfied. In fact,
\begin{equation*}
\label{commutation with the operator equation 1}
\begin{split}
[A,P_{j,0}]=0 \mod \Psi^{-1} 
\quad 
\Longleftrightarrow 
\quad 
[A,P_{j,0}]_\mathrm{prin}=0
\end{split}
\end{equation*}
and
\begin{equation*}
\label{commutation with the operator equation 2}
\begin{split}
[A,P_{j,0}]_\mathrm{prin}
&
=
A_\mathrm{prin}P^{(j)} - P^{(j)}A_\mathrm{prin}
\\
&
=
\sum_l h^{(l)}[P^{(l)},P^{(j)}]
=
0.
\end{split}
\end{equation*}

Let $\widetilde{R}_{j,k}\in \Psi^{-k}$, $\widetilde{R}_{j,k}=\widetilde{R}_{j,k}^* \mod \Psi^{-\infty}$, be such that
\begin{equation}
\label{commutation with the operator equation 3}
[\widetilde{R}_{j,k}]_\mathrm{prin}=[\widetilde{X}_{j,k}]_\mathrm{prin}+\frac12\sum_{l\ne j} (R_{j,l,k} +R_{l,j,k}),
\end{equation}
see~\eqref{Construction of a single pseudodifferential projection equation 14}, \eqref{Construction of a single pseudodifferential projection equation 16} and \eqref{Construction of a single pseudodifferential projection equation 21},
and define
\begin{equation}
\label{commutation with the operator equation 4}
M_{j,k}:=[A,P_{j,k-1}+\widetilde{R}_{j,k}]_\mathrm{prin,k-s}.
\end{equation}

Then, for $k\ge 1$, satisfying \eqref{commutation with operator recursive condition} reduces to determining $Z_{j,l,k}$ such that
\begin{equation}
\label{commutation with the operator equation 5}
\sum_{l\ne j}(h^{(j)}-h^{(l)})[Z_{j,l,k}+Z_{l,j,k}]=-M_{j,k},
\end{equation}
\begin{equation}
\label{commutation with the operator equation 6}
Z_{j,l,k}=P^{(j)}Z_{j,l,k}P^{(l)}, \qquad Z_{j,l,k}=-Z_{l,j,k}^*.
\end{equation}

\begin{lemma}
\label{lemma commutation with the operator}
A solution to \eqref{commutation with the operator equation 5}, \eqref{commutation with the operator equation 6} is given by
\begin{equation}
\label{lemma commutation with the operator equation 1}
Z_{j,l,k}=-\frac{P^{(j)}M_{j,k}P^{(l)}}{h^{(j)}-h^{(l)}}.
\end{equation}
\end{lemma}

\begin{proof}
Formula \eqref{lemma commutation with the operator equation 1} clearly satisfies \eqref{commutation with the operator equation 6}, because $M_{j,k}$ is skew-Hermitian.

It is easy to see that necessary solvability conditions are
\begin{equation}
\label{proof lemma commutation with the operator 1}
P^{(j)}M_{j,k}P^{(j)}=0
\end{equation}
and
\begin{equation}
\label{proof lemma commutation with the operator 2}
P^{(r)}M_{j,k}P^{(n)}=0 \quad\text{for}\quad r,n \ne j.
\end{equation}
Let us show that \eqref{proof lemma commutation with the operator 1} and \eqref{proof lemma commutation with the operator 2} are satisfied. We have
\begin{equation*}
\label{proof lemma commutation with the operator 3}
\begin{split}
P^{(j)}M_{j,k}P^{(j)}
&
=
P^{(j)}[A,P_{j,k-1}+\widetilde{R}_{j,k}]_\mathrm{prin,k-s}P^{(j)}
\\
&
=
[P_{j,k-1}+\widetilde{R}_{j,k}]_{\mathrm{prin},0}[A,P_{j,k-1}+\widetilde{R}_{j,k}]_\mathrm{prin,k-s}[P_{j,k-1}+\widetilde{R}_{j,k}]_{\mathrm{prin,0}}
\\
&
=
0.
\end{split}
\end{equation*}
In the above argument we used the fact that
\begin{equation}
\label{proof lemma commutation with the operator 8}
(P_{j,k-1}+\widetilde{R}_{j,k})^2=P_{j,k-1}+\widetilde{R}_{j,k} \mod \Psi^{-k-1},
\end{equation}
as established in subsection~\ref{Construction of a single pseudodifferential projection}.
Similarly, we have
\begin{equation*}
\label{proof lemma commutation with the operator 3bis}
\begin{split}
P^{(r)}M_{j,k}P^{(n)}
&
=
P^{(r)}[A,P_{j,k-1}+\widetilde{R}_{j,k}]_\mathrm{prin,k-s}P^{(r)}
\\
&
=
[P_{r,k-1}+\widetilde{R}_{r,k}]_{\mathrm{prin},0}[A,P_{j,k-1}+\widetilde{R}_{j,k}]_\mathrm{prin,k-s}[P_{n,k-1}+\widetilde{R}_{n,k}]_{\mathrm{prin,0}}
\\
&
=
0.
\end{split}
\end{equation*}
In the above argument we used the fact that
\begin{equation}
\label{proof lemma commutation with the operator 9}
(P_{j,k-1}+\widetilde{R}_{j,k})(P_{l,k-1}+\widetilde{R}_{l,k})=(P_{l,k-1}+\widetilde{R}_{l,k})(P_{j,k-1}+\widetilde{R}_{j,k})=0 \mod \Psi^{-k-1},
\end{equation}
as established in subsection~\ref{Construction of a basis of pseudodifferential projections}.

Formulae \eqref{proof lemma commutation with the operator 1} and \eqref{proof lemma commutation with the operator 2} imply
\begin{equation}
\label{proof lemma commutation with the operator 5}
M_{j,k}=\sum_{l\ne j}(P^{(j)}M_{j,k}P^{(l)}+P^{(l)}M_{j,k}P^{(j)}).
\end{equation}
Furthermore, the matrix-functions $M_{j,k}$ satisfy the identity
\begin{equation}
\label{proof lemma commutation with the operator 6}
P^{(l)}M_{l,k}P^{(j)}=-P^{(l)}M_{j,k}P^{(j)} \quad \text{for}\quad j\ne l.
\end{equation}
In fact, for $j\ne l$ we have
\begin{equation*}
\label{proof lemma commutation with the operator 7}
\begin{split}
P^{(l)}M_{l,k}P^{(j)}
&
=
P^{(l)}[A,P_{l,k-1}+\widetilde{R}_{l,k}]_\mathrm{prin,k-s}P^{(j)}
\\
&
=
[P_{l,k-1}+\widetilde{R}_{l,k}]_\mathrm{prin,0}[A(P_{l,k-1}+\widetilde{R}_{l,k})-(P_{l,k-1}+\widetilde{R}_{l,k})A]_\mathrm{prin,k-s}[P_{j,k-1}+\widetilde{R}_{j,k}]_\mathrm{prin,0}
\\
&
=-[(P_{l,k-1}+\widetilde{R}_{l,k})A(P_{j,k-1}+\widetilde{R}_{j,k}])]_\mathrm{prin,k-s}
\\
&
=
-[P_{l,k-1}+\widetilde{R}_{l,k}]_\mathrm{prin,0}[A(P_{j,k-1}+\widetilde{R}_{j,k})-(P_{j,k-1}+\widetilde{R}_{j,k})A]_\mathrm{prin,k-s}[P_{j,k-1}+\widetilde{R}_{j,k}]_\mathrm{prin,0}
\\
&
=
-P^{(l)}[A,P_{j,k-1}+\widetilde{R}_{j,k}]_\mathrm{prin,k-s}P^{(j)}
\\
&
=
-
P^{(l)}M_{j,k}P^{(j)}.
\end{split}
\end{equation*}
In the above argument we used \eqref{proof lemma commutation with the operator 8} and \eqref{proof lemma commutation with the operator 9}.

It remains only to substitute \eqref{lemma commutation with the operator equation 1} into \eqref{commutation with the operator equation 5} and use \eqref{proof lemma commutation with the operator 5}--\eqref{proof lemma commutation with the operator 6}.
\end{proof}

\begin{proposition}
\label{proposition about uniqueness}
The solution found in Lemma~\ref{lemma commutation with the operator} is the unique solution to \eqref{commutation with the operator equation 5}--\eqref{commutation with the operator equation 6}.
\end{proposition}

\begin{proof}
It suffices to show that the homogeneous system does not admit nontrivial solutions.

Suppose $\widetilde{Z}_{j,l}$ is a solution of the homogeneous system
\begin{equation}
\label{proof theorem about uniqueness equation 1}
\sum_{l\ne j}(h^{(j)}-h^{(l)})[\widetilde{Z}_{j,l}+\widetilde{Z}_{l,j}]=0.
\end{equation}
Then, multiplying \eqref{proof theorem about uniqueness equation 1} by $P^{(j)}$ on the left and by $P^{(n)}$, $n\ne j$, on the right, we obtain
\begin{equation*}
\label{proof theorem about uniqueness equation 2}
(h^{(j)}-h^{(n)}) \widetilde{Z}_{j,n}=0.
\end{equation*}
\end{proof}

Combining Theorem~\ref{theorem basis of pseudodifferential projections}, Lemma~\ref{lemma commutation with the operator} and Proposition~\ref{proposition about uniqueness} we obtain Theorem~\ref{Main results theorem 1}.

\subsection{The algorithm}
\label{The algorithm}

Let us summarise the results from subsections~\ref{Construction of a single pseudodifferential projection}--\ref{Commutation with an elliptic operator} in the form of a concise algorithm for the construction of the full symbol of pseudodifferential projections.

\

\textbf{Step 1.} Given the $m$ eigenprojections $P^{(j)}(x,\xi)$ of $A_\mathrm{prin}\,$, choose $m$ arbitrary pseudo\-differential operators $P_{j,0}\in \Psi^0$ satisfying
\begin{enumerate}[(i)]
\item
$(P_{j,0})_\mathrm{prin}=P^{(j)}$,

\item 
$P_{j,0}=P_{j,0}^* \mod \Psi^{-\infty}$.
\end{enumerate}

\

\textbf{Step 2.} For $k=1,2,\ldots$ define
\begin{equation*}
\label{algorithm equation 1}
P_{j,k}:=P_{j,0}+\sum_{n=1}^k X_{j,n}, \qquad X_{j,n}\in \Psi^{-n}.
\end{equation*}

Assuming we have determined the pseudodifferential operator $P_{j,k-1}$, compute, one after the other, the following quantities:
\begin{enumerate}[(a)]
\item 
$
R_{j,k}=-((P_{j,k-1})^2-P_{j,k-1})_\mathrm{prin,k}\,,
$

\item
$
S_{j,k}=-R_{j,k}+P^{(j)}R_{j,k}+R_{j,k}P^{(j)},
$

\item
$
V_{j,l,k}=-\frac12  ( 
(P_{j,k-1}P_{l,k-1})_\mathrm{prin,k}
+P^{(j)}S_{l,k}+ S_{j,k}P^{(l)}
),
$

\item
$
Z_{j,l,k}=
(h^{(l)}-h^{(j)})^{-1}
P^{(j)}
\left(
[A\,,P_{j,k-1}]_\mathrm{prin,k-s}
+
[A_\mathrm{prin}\,, S_{j,k}+\sum_{n\ne j}(V_{j,n,k}+V_{n,j,k})]
\right)
P^{(l)},
$
\end{enumerate}
for $l\ne j$.

\

\textbf{Step 3.} Choose a pseudodifferential operator $X_{j,k}\in \Psi^{-k}$ satisfying
\begin{enumerate}[(i)]
\item
$
(X_{j,k})_\mathrm{prin}=S_{j,k}+\sum_{l\ne j}
(V_{j,l,k}+V_{l,j,k})+\sum_{l\ne j}(Z_{j,l,k}-Z_{l,j,k}),
$

\item
$
X_{j,k}=X_{j,k}^* \mod \Psi^{-\infty}.
$
\end{enumerate}

\

\textbf{Step 4.}
Put
\begin{equation*}
\label{algorithm equation 2}
P_j \sim P_{j,0}+\sum_{n=1}^{+\infty}X_{j,n}\,.
\end{equation*}

\subsection{Proof of Theorem {\ref{Main results theorem 1}}}
\label{Proof of Theorem 2.1}

\begin{proof}
In subsections~\ref{Construction of a single pseudodifferential projection}--\ref{Commutation with an elliptic operator} we established existence of our pseudodifferential projections. It remains to prove that they are unique.

Suppose there exist two sets of pseudodifferential projections, $P_j$ and $P_j'$, satisfying 
\eqref{definition of pseudodifferential projection condition 2}--\eqref{definition of pseudodifferential projection condition 1bis}
whose difference is not in $\Psi^{-\infty}$. Then there exists a natural number $k$ such that
\[
P_j-P_j'\in\Psi^{-k}\quad\forall j,
\]
but
\[
P_j-P_j'\not\in\Psi^{-k-1}\quad\text{for some}\quad j.
\]
Consider the pseudodifferential operators
\[
B_j:=P_j-P_j'\in\Psi^{-k}.
\]

The fact that the operators $P_j=B_j+P_j'$ satisfy \eqref{definition of pseudodifferential projection condition 2}--\eqref{definition of pseudodifferential projection condition 1bis} yields the following system of equations for $(B_j)_\mathrm{prin}$:
\begin{equation}
\label{uniqueness equation 1}
P^{(j)}(B_j)_\mathrm{prin}
+
(B_j)_\mathrm{prin}P^{(j)}
-
(B_j)_\mathrm{prin}
=0,
\end{equation}
\begin{equation}
\label{uniqueness equation 2}
(B_j)_\mathrm{prin}^*
=
(B_j)_\mathrm{prin},
\end{equation}
\begin{equation}
\label{uniqueness equation 3}
P^{(l)}(B_j)_\mathrm{prin}
+
(B_l)_\mathrm{prin}P^{(j)}
=0,\qquad l\ne j,
\end{equation}
\begin{equation}
\label{uniqueness equation 4}
A_\mathrm{prin}(B_j)_\mathrm{prin}
-
(B_j)_\mathrm{prin}A_\mathrm{prin}
=0.
\end{equation}

The matrix-function $(B_j)_\mathrm{prin}$ can be uniquely represented in the form
\begin{equation}
\label{uniqueness equation 5}
(B_j)_\mathrm{prin}
=
\sum_{l,n}B_{j,l,n},
\end{equation}
where the matrix-functions $B_{j,l,n}$ satisfy
\begin{equation}
\label{uniqueness equation 6}
B_{j,l,n}=P^{(l)}B_{j,l,n}P^{(n)}.
\end{equation}
Substituting \eqref{uniqueness equation 5} into \eqref{uniqueness equation 1} we immediately get
\begin{equation*}
\label{uniqueness equation 7}
B_{j,j,j}=0,
\end{equation*}
\begin{equation*}
\label{uniqueness equation 8}
B_{j,l,n}=0\quad\text{for}\quad l\ne j,\ n\ne j,
\end{equation*}
so that \eqref{uniqueness equation 5} can be equivalently recast as
\begin{equation}
\label{uniqueness equation 8bis}
(B_j)_\mathrm{prin}=\sum_{l\ne j}(B_{j,j,l}+B_{j,l,j}).
\end{equation}

Substituting
\eqref{uniqueness equation 8bis} into \eqref{uniqueness equation 4}
and taking into account
\eqref{uniqueness equation 6}
we get
\begin{equation*}
\label{uniqueness equation 14}
\sum_{l\ne j}
(h^{(j)}-h^{(l)})(B_{j,j,l}-B_{j,l,j})
=0.
\end{equation*}
But the latter only admits the trivial solution.
\end{proof}

\section{Commutation with an elliptic operator: revisited}
\label{Commutation with an elliptic operator: revisited Dima}

\subsection{An abstract theorem on pseudodifferential  projections}
\label{An abstract theorem on pseudodifferential  projections}

The argument presented in subsection~\ref{Proof of Theorem 2.1}
shows that conditions \eqref{definition of pseudodifferential projection condition 1bis}, \eqref{definition of pseudodifferential projection condition 2} and \eqref{commutation condition}
alone force uniqueness of our orthonormal pseudodifferential basis commuting with $A$. This observation
motivates us to formulate the following abstract result.

\begin{theorem}
\label{theorem minimal assumptions}
Suppose that we are given $m$ pseudodifferential operators $P_j\in\Psi^0$ satisfying conditions
\begin{equation}
\label{definition of pseudodifferential projection condition 2 with j}
P_j^2=P_j\mod\Psi^{-\infty}
\end{equation}
and
\eqref{commutation condition},
\eqref{definition of pseudodifferential projection condition 1bis}.
Then these operators satisfy
\begin{equation}
\label{definition of pseudodifferential projection condition 3 with j}
P_j^*=P_j\mod\Psi^{-\infty}
\end{equation}
and
\eqref{definition of system of pseudodifferential projections condition 1},
\eqref{definition of system of pseudodifferential projections condition 2}.
\end{theorem}

\begin{proof}
To begin with, let us show that we have \eqref{definition of pseudodifferential projection condition 3 with j}. 
Suppose there exists a $j$ such that, for some natural $k$, $P_j-P_j^*\in \Psi^{-k}$ but
\begin{equation}
\label{proof theorem minimal assumptions equation 01}
P_j-P_j^*\not\in \Psi^{-k-1}.
\end{equation}
Put
\begin{equation}
\label{proof theorem minimal assumptions equation 02}
B:=P_j-P_j^*.
\end{equation}
Then conditions \eqref{definition of pseudodifferential projection condition 2 with j}, \eqref{commutation condition} and \eqref{definition of pseudodifferential projection condition 1bis} give us the following equations for $B_\mathrm{prin}$:
\begin{equation}
\label{proof theorem minimal assumptions equation 03}
B_\mathrm{prin}P^{(j)}+P^{(j)}B_\mathrm{prin}-B_\mathrm{prin}=0,
\end{equation}
\begin{equation}
\label{proof theorem minimal assumptions equation 04}
A_\mathrm{prin}B_\mathrm{prin}-B_\mathrm{prin}A_\mathrm{prin}=0.
\end{equation}
Formula \eqref{proof theorem minimal assumptions equation 03} implies 
\begin{equation*}
\label{proof theorem minimal assumptions equation 05}
P^{(j)}B_\mathrm{prin}P^{(j)}=0,
\end{equation*}
so that we have
\begin{equation}
\label{proof theorem minimal assumptions equation 06}
B_\mathrm{prin}=\sum_{l\ne j}(P^{(l)}B_\mathrm{prin}P^{(j)}+P^{(j)}B_\mathrm{prin}P^{(l)})\,.
\end{equation}
Substituting \eqref{proof theorem minimal assumptions equation 06} into \eqref{proof theorem minimal assumptions equation 04} we get
\begin{equation*}
\label{proof theorem minimal assumptions equation 07}
\sum_{l\ne j}(h^{(j)}-h^{(l)})(P^{(j)}B_\mathrm{prin}P^{(l)}-P^{(l)}B_\mathrm{prin}P^{(j)})=0,
\end{equation*}
which implies 
\begin{equation*}
\label{proof theorem minimal assumptions equation 08}
B_\mathrm{prin}=0.
\end{equation*}
But the latter contradicts \eqref{proof theorem minimal assumptions equation 01}.

Next, let us show that \eqref{definition of system of pseudodifferential projections condition 1} holds. Arguing by contradiction, suppose there exist $j$ and~$l$, $j\ne l$, such that $P_j$ and $P_l$ satisfy the assumptions of the theorem and, for some natural $k$, $P_j P_l \in \Psi^{-k}$ but
\begin{equation}
\label{proof theorem minimal assumptions equation 2}
P_j P_l \not \in \Psi^{-k-1}.
\end{equation}
Put $C:=P_j P_l$. Then conditions \eqref{definition of pseudodifferential projection condition 2 with j}, \eqref{commutation condition} and \eqref{definition of pseudodifferential projection condition 1bis} give us the following constraints for $C_\mathrm{prin}$:
\begin{equation}
\label{proof theorem minimal assumptions equation 3}
P^{(j)}C_\mathrm{prin}-C_\mathrm{prin}=0,
\end{equation}
\begin{equation}
\label{proof theorem minimal assumptions equation 4}
C_\mathrm{prin}P^{(l)}-C_\mathrm{prin}=0,
\end{equation}
\begin{equation}
\label{proof theorem minimal assumptions equation 5}
A_\mathrm{prin}C_\mathrm{prin}-C_\mathrm{prin}A_\mathrm{prin}=0.
\end{equation}
Formulae \eqref{proof theorem minimal assumptions equation 3} and \eqref{proof theorem minimal assumptions equation 4} imply
\begin{equation}
\label{proof theorem minimal assumptions equation 6}
C_\mathrm{prin}=P^{(j)}C_\mathrm{prin}P^{(l)}.
\end{equation}
Substituting \eqref{proof theorem minimal assumptions equation 6} into \eqref{proof theorem minimal assumptions equation 5} we obtain
\begin{equation*}
\label{proof theorem minimal assumptions equation 7}
(h^{(j)}-h^{(l)})C_\mathrm{prin}=0
\end{equation*}
which, in turn, yields
\begin{equation*}
\label{proof theorem minimal assumptions equation 8}
C_\mathrm{prin}=0.
\end{equation*}
The latter contradicts \eqref{proof theorem minimal assumptions equation 2}.

The fact that we have \eqref{definition of system of pseudodifferential projections condition 2} now follows from Theorem~\ref{theorem identity operator}.
\end{proof}

Note that an alternative proof of Theorem~\ref{theorem minimal assumptions}
can be obtained by arguing as in subsection~\ref{Proof of Theorem 2.1}
and using the constructive proof of the existence of projections $P_j$ provided in
subsections \ref{Construction of a single pseudodifferential projection}--\ref{Commutation with an elliptic operator}.
However, we decided to give here an abstract self-contained
proof which does not rely on the explicit construction of the $P_j$'s.

\subsection{Developing a simplified algorithm}
\label{Developing a simplified algorithm}

Theorem~\ref{theorem minimal assumptions} opens the way to the formulation of a new, shorter version of our algorithm. 

In this subsection we will show directly what was argued at the end of subsection~\ref{Proof of Theorem 2.1}, namely that a family of orthogonal pseudodifferential projections in the sense of Definition~\ref{definition of pseudodifferential projection} satisfying the commutation relation \eqref{commutation condition} is automatically an orthonormal pseudodifferential basis in the sense of Definition~\ref{definition of system of pseudodifferential projections}.

Suppose that, in accordance with subsection~\ref{Construction of a single pseudodifferential projection}, we have constructed $m$ orthogonal pseudodifferential projections of the form
\begin{equation}
\label{commutation revisited equation 1}
P_j\sim P_{j,0}+\sum_{k=1}^{+\infty} X_{j,k}
\end{equation}
where $(X_{j,k})_\mathrm{prin}$ is given by
\begin{equation}
\label{commutation revisited equation 3}
(X_{j,k})_\mathrm{prin}=-R_{j,k}+P^{(j)}R_{j,k}+R_{j,k}P^{(j)}+\sum_{l\ne j}(Y_{j,l,k}+Y_{j,l,k}^*),
\end{equation}
\begin{equation}
\label{commutation revisited equation 4}
Y_{j,l,k}=P^{(j)}Y_{j,l,k}P^{(l)},\qquad Y_{j,l,k}^*=P^{(l)}Y_{j,l,k}^*P^{(j)},
\end{equation}
cf.~\eqref{Construction of a single pseudodifferential projection equation 12} and \eqref{Construction of a single pseudodifferential projection equation 13}.

Choose pseudodifferential operators $\widetilde X_{j,k}\in \Psi^{-k}$, $\widetilde X_{j,k}=\widetilde X_{j,k}^* \mod \Psi^{-\infty}$, such that
\begin{equation}
\label{commutation revisited equation 5}
(\widetilde X_{j,k})_\mathrm{prin}=-R_{j,k}+P^{(j)}R_{j,k}+R_{j,k}P^{(j)}
\end{equation}
and put
\begin{equation}
\label{commutation revisited equation 6}
T_{j,k}:=-[A,P_{j,k-1}+\widetilde X_{j,k}]_{\mathrm{prin},k-s}.
\end{equation}
Then, in view of \eqref{commutation revisited equation 3} and \eqref{commutation revisited equation 4}, satisfying \eqref{commutation condition} reduces to solving the system of equations
\begin{equation}
\label{commutation revisited equation 7}
\sum_{l\ne j} (h^{(j)}-h^{(l)})(Y_{j,l,k}-Y_{j,l,k}^*)=T_{j,k}
\end{equation}
for $Y_{j,l,k}$.

\begin{lemma}
\label{lemma commutation revisited}
The unique solution to \eqref{commutation revisited equation 7} is given by
\begin{equation}
\label{lemma commutation revisited equation 1}
Y_{j,l,k}=\frac{P^{(j)}T_{j,k}P^{(l)}}{h^{(j)}-h^{(l)}}.
\end{equation}
\end{lemma}
\begin{proof}
It is easy to see that necessary solvability conditions are
\begin{equation}
\label{proof lemma commutation revisited equation 1}
P^{(j)}T_{j,k}P^{(j)}=0,
\end{equation}
\begin{equation}
\label{proof lemma commutation revisited equation 2}
T_{j,k}^*=-T_{j,k}.
\end{equation}
Condition \eqref{proof lemma commutation revisited equation 2} is clearly satisfied. Let us check that the same is true for \eqref{proof lemma commutation revisited equation 1}. We have
\begin{equation*}
\label{proof lemma commutation revisited equation 3}
\begin{split}
P^{(j)}T_{j,k}P^{(j)}
&
=
-P^{(j)}[A,P_{j,k-1}+\widetilde X_{j,k}]_{\mathrm{prin},k-s} P^{(j)}
\\
&
=
-(P_{j,k-1}+\widetilde X_{j,k})_\mathrm{prin,0}[A,P_{j,k-1}+\widetilde X_{j,k}]_{\mathrm{prin},k-s} (P_{j,k-1}+\widetilde X_{j,k})_\mathrm{prin,0}
\\
&
=
0.
\end{split}
\end{equation*}
In the last step of the above calculation we used the fact that
\begin{equation}
\label{proof lemma commutation revisited equation 4}
(P_{j,k-1}+\widetilde X_{j,k})^2=P_{j,k-1}+\widetilde X_{j,k} \mod \Psi^{-k-1},
\end{equation}
which was established in subsection~\ref{Construction of a single pseudodifferential projection}.

Formula \eqref{proof lemma commutation revisited equation 1} implies
\begin{equation}
\label{proof lemma commutation revisited equation 5}
T_{j,k}=\sum_{l\ne j} (P^{(j)}T_{j,k}P^{(l)}+P^{(l)}T_{j,k}P^{(j)}).
\end{equation}

By substituting \eqref{lemma commutation revisited equation 1} into \eqref{commutation revisited equation 7} and taking into account \eqref{proof lemma commutation revisited equation 2} and \eqref{proof lemma commutation revisited equation 5}
one shows that \eqref{lemma commutation revisited equation 1} is a solution. 

To complete the proof it remains only to observe that the homogeneous system
\begin{equation*}
\sum_{l\ne j} (h^{(j)}-h^{(l)})(Y_{j,l,k}-Y_{j,l,k}^*)=0
\end{equation*}
complemented with \eqref{commutation revisited equation 4} admits only the trivial solution.
\end{proof}

Lemma~\ref{lemma commutation revisited} completely determines $({X}_{j,k})_\mathrm{prin}$. Let us show that for $j\ne l$ we automatically have
\begin{equation}
\label{commutation revisited equation 8}
P_{j,k}P_{l,k} \in \Psi^{-k-1}.
\end{equation}
In view of our recursive construction, condition \eqref{commutation revisited equation 8} can be equivalently rewritten as
\begin{equation}
\label{commutation revisited equation 11}
\frac{P^{(j)}(T_{l,k}+T_{j,k})P^{(l)}}{h^{(j)}-h^{(l)}}=-((P_{j,k-1}+\widetilde{X}_{j,k})(P_{l,k-1}+\widetilde{X}_{l,k}))_\mathrm{prin,k}.
\end{equation}

In view of \eqref{proof lemma commutation revisited equation 4}, we have
\begin{equation}
\label{commutation revisited equation 12}
\begin{split}
P^{(j)}(T_{l,k}+T_{j,k})P^{(l)}
=
&
-(P_{j,k-1}+\widetilde X_{j,k})_\mathrm{prin,0}[A,P_{l,k-1}+\widetilde X_{l,k}]_{\mathrm{prin},k-s}(P_{l,k-1}+\widetilde X_{l,k})_\mathrm{prin,0}
\\
&
-(P_{j,k-1}+\widetilde X_{j,k})_\mathrm{prin,0}[A,P_{j,k-1}+\widetilde X_{j,k}]_{\mathrm{prin},k-s}(P_{l,k-1}+\widetilde X_{l,k})_\mathrm{prin,0}
\\
=
&
-P^{(j)}[A,(P_{j,k-1}+\widetilde X_{j,k})(P_{l,k-1}+\widetilde X_{l,k})]_\mathrm{prin,k-s}P^{(l)}
\\
=
&
-P^{(j)}[A_\mathrm{prin},((P_{j,k-1}+\widetilde X_{j,k})(P_{l,k-1}+\widetilde X_{l,k}))_\mathrm{prin,k}]P^{(l)}
\\
=
&
-(h^{(j)}-h^{(l)})((P_{j,k-1}+\widetilde X_{j,k})(P_{l,k-1}+\widetilde X_{l,k}))_\mathrm{prin,k}.
\end{split}
\end{equation}
Substituting \eqref{commutation revisited equation 12} into \eqref{commutation revisited equation 11}, we see that \eqref{commutation revisited equation 11} is satisfied.

The fact that the operators \eqref{commutation revisited equation 1}, \eqref{commutation revisited equation 3}, \eqref{lemma commutation with the operator equation 1} satisfy \eqref{definition of system of pseudodifferential projections condition 2} now follows from Theorem~\ref{theorem identity operator}.

\subsection{The simplified algorithm}
\label{The simplified algorithm}

The construction from the previous subsection can be summarised as follows.

\

\textbf{Step 1.} Given the $m$ eigenprojections $P^{(j)}$ of $A_\mathrm{prin}\,$, choose $m$ arbitrary pseudo\-differential operators $P_{j,0}\in \Psi^0$ satisfying $(P_{j,0})_\mathrm{prin}=P^{(j)}$.

\

\textbf{Step 2.} For $k=1,2,\ldots$ define
\begin{equation}
\label{simplified algorithm equation 1}
P_{j,k}:=P_{j,0}+\sum_{n=1}^k X_{j,n}, \qquad X_{j,n}\in \Psi^{-n}.
\end{equation}

Assuming we have determined the pseudodifferential operator $P_{j,k-1}$, compute, one after the other, the following quantities:
\begin{enumerate}[(a)]
\item 
$
R_{j,k}:=-((P_{j,k-1})^2-P_{j,k-1})_\mathrm{prin,k}\,,
$

\item 
$
S_{j,k}:=-R_{j,k}+P^{(j)}R_{j,k}+R_{j,k}P^{(j)}\,,
$

\item
$
T_{j,k}:=[P_{j,k-1},A]_{\mathrm{prin},k-s}+[S_{j,k},A_\mathrm{prin}].
$
\end{enumerate}

\

\textbf{Step 3.} Choose a pseudodifferential operator $X_{j,k}\in \Psi^{-k}$ satisfying
\begin{equation}
\label{revised algorithm crucial formula}
(X_{j,k})_\mathrm{prin}=S_{j,k}+\sum_{l\ne j}\frac{P^{(j)}T_{j,k}P^{(l)}-P^{(l)}T_{j,k}P^{(j)}}{h^{(j)}-h^{(l)}}.
\end{equation}

\

\textbf{Step 4.}
Put
\begin{equation}
\label{simplified algorithm equation 2}
P_j \sim P_{j,0}+\sum_{n=1}^{+\infty}X_{j,n}\,.
\end{equation}

\begin{remark}
We now have at our disposal two versions of the construction algorithm: the one just given above and that presented in subsection~\ref{The algorithm}. Each has its own advantages. 

The one given here is more concise and requires relatively simple calculations to obtain the final formulae. We will have an opportunity to appreciate this in Section~\ref{Subprincipal symbol of pseudodifferential projections}, when we will compute the subprincipal symbol of pseudodifferential projections. 

Our original algorithm given in subsection~\ref{The algorithm}, despite being longer and entailing more complicated calculations, provides more refined information. On the one hand, it allows one to single out 
\begin{enumerate}[(i)]

\item
contributions to the symbols making $P_j$ an orthogonal projection (the $S_{j,k}$'s),

\item
those responsible for the orthonormality condition (the $V_{j,l,k}$'s) and 

\item
those ensuring the commutation with $A$ (the $Z_{j,l,k}$'s).
\end{enumerate}
On the other hand, it allows one to keep track of how many degrees of freedom are left in the symbol at each stage of the construction algorithm, at the same time shedding some light on how the available degrees of freedom are used up.
This makes the first version of the algorithm more suitable if one is only interested in constructing a single pseudodifferential projection or a pseudodifferential basis, without necessarily relating the $P_j$'s to an elliptic operator $A$.

\end{remark}

\section{Subprincipal symbol of pseudodifferential projections}
\label{Subprincipal symbol of pseudodifferential projections}

In this section we will carry out the first iteration of the above algorithm explicitly, to obtain a closed formula for the subprincipal symbol of our projections and prove Theorem~\ref{Main results theorem 2}.

Let us choose the initial operators $P_{j,0}\in \Psi^0$ satisfying the additional property
\begin{equation}
\label{subprincipal symbol of projections equation 1}
(P_{j,0})_\mathrm{sub}=0.
\end{equation}
Then \eqref{subprincipal symbol of projections equation 1} and \eqref{simplified algorithm equation 2} imply
\begin{equation*}
(P_j)_\mathrm{sub}=(X_{j,1})_\mathrm{prin}.
\end{equation*}
The task at hand reduces to computing $(X_{j,1})_\mathrm{prin}$.

Recall that the
formula for the subprincipal of a composition of pseudodifferential operators reads  \cite[Eqn.~(1.4)]{DuGu}
\begin{equation}
\label{subprincipal symbol composition}
[BC]_\mathrm{sub}=
B_\mathrm{prin}C_\mathrm{sub}
+
B_\mathrm{sub}C_\mathrm{prin}
+\frac i2
\{B_\mathrm{prin},C_\mathrm{prin}\}.
\end{equation}
Note that \cite{DuGu} adopts the opposite sign convention for the Poisson bracket.

Using \eqref{subprincipal symbol of projections equation 1} and \eqref{subprincipal symbol composition}, we obtain
\begin{equation}
\label{subprincipal symbol of projections equation 2}
\begin{split}
R_{j,1}
&
=
-\frac{i}{2}\{P^{(j)}, P^{(j)}\}.
\end{split}
\end{equation}
Formula \eqref{subprincipal symbol of projections equation 2} gives us
\begin{equation}
\label{subprincipal symbol of projections equation 3}
S_{j,1}=\frac{i}{2}(\{P^{(j)}, P^{(j)}\}-P^{(j)}\{P^{(j)}, P^{(j)}\}-\{P^{(j)}, P^{(j)}\}P^{(j)}).
\end{equation}

In view of \eqref{subprincipal symbol composition} and with account of \eqref{subprincipal symbol of projections equation 1}, we have
\begin{equation}
\label{subprincipal symbol of projections equation 4}
\begin{split}
[P_{j,0},A]_{1-s}
&
=
(P_{j,0}A-AP_{j,0})_\mathrm{sub}
\\
&
=
P^{(j)}A_\mathrm{sub}+\frac{i}2\{P^{(j)},A_\mathrm{prin}\}-A_\mathrm{sub}P^{(j)}-\frac{i}2\{A_\mathrm{prin},P^{(j)}\}
\\
&
=
[P^{(j)},A_\mathrm{sub}] +\frac i2(\{P^{(j)},A_\mathrm{prin}\}-\{A_\mathrm{prin},P^{(j)}\}).
\end{split}
\end{equation}

Observe that
\begin{equation*}
\label{subprincipal symbol of projections equation 5}
P^{(j)}\{P^{(j)}, P^{(j)}\}=\{P^{(j)}, P^{(j)}\}P^{(j)}=P^{(j)}\{P^{(j)}, P^{(j)}\}P^{(j)}.
\end{equation*}
Therefore, by means of \eqref{subprincipal symbol of projections equation 3}, we obtain
\begin{equation}
\label{subprincipal symbol of projections equation 6}
\begin{split}
[S_{j,1},A_\mathrm{prin}]
&
=
\frac{i}2[\{P^{(j)}, P^{(j)}\}-P^{(j)}\{P^{(j)}, P^{(j)}\}-\{P^{(j)}, P^{(j)}\}P^{(j)},A_\mathrm{prin}]
\\
&
=
\frac i2(\{P^{(j)}, P^{(j)}\}A_\mathrm{prin}-A_\mathrm{prin}\{P^{(j)}, P^{(j)}\})\,.
\end{split}
\end{equation}

By combining \eqref{subprincipal symbol of projections equation 4} and \eqref{subprincipal symbol of projections equation 6} we get
\begin{equation*}
\label{subprincipal symbol of projections equation 7}
\begin{split}
T_{j,1}
&
=
\frac i2(\{P^{(j)},A_\mathrm{prin}\}-\{A_\mathrm{prin},P^{(j)}\}+[\{P^{(j)}, P^{(j)}\},A_\mathrm{prin}])
+[P^{(j)},A_\mathrm{sub}].
\end{split}
\end{equation*}
We can then compute
\begin{equation}
\label{subprincipal symbol of projections equation 8}
\begin{split}
P^{(j)}T_{j,1}P^{(l)}
&
=
\frac i2(P^{(j)}\{P^{(j)},A_\mathrm{prin}\}P^{(l)}-P^{(j)}\{A_\mathrm{prin},P^{(j)}\}P^{(l)})
\\
&
+\frac i2(h^{(l)}-h^{(j)})P^{(j)}\{P^{(j)}, P^{(j)}\}P^{(l)}
+
P^{(j)}A_\mathrm{sub}P^{(l)}
\\
&
=
\frac i2(P^{(j)}\{P^{(j)},A_\mathrm{prin}\}P^{(l)}-P^{(j)}\{A_\mathrm{prin},P^{(j)}\}P^{(l)})
+
P^{(j)}A_\mathrm{sub}P^{(l)}.
\end{split}
\end{equation}
In the above calculation we used that 
\begin{equation*}
\label{subprincipal symbol of projections equation 9}
P^{(j)}\{P^{(j)}, P^{(j)}\}P^{(l)}=0 \quad \text{for}\quad j\ne l.
\end{equation*}

Similarly, we have
\begin{equation}
\label{subprincipal symbol of projections equation 10}
\begin{split}
P^{(l)}T_{j,1}P^{(j)}
&
=
\frac i2(P^{(l)}\{P^{(j)},A_\mathrm{prin}\}P^{(j)}-P^{(l)}\{A_\mathrm{prin},P^{(j)}\}P^{(l)})
\\
&
+\frac i2(h^{(j)}-h^{(l)})P^{(l)}\{P^{(j)}, P^{(j)}\}P^{(j)}
-
P^{(l)}A_\mathrm{sub}P^{(j)}
\\
&
=
\frac i2(P^{(l)}\{P^{(j)},A_\mathrm{prin}\}P^{(j)}-P^{(l)}\{A_\mathrm{prin},P^{(j)}\}P^{(j)})
-
P^{(l)}A_\mathrm{sub}P^{(j)}\,.
\end{split}
\end{equation}

Substituting \eqref{subprincipal symbol of projections equation 3}, \eqref{subprincipal symbol of projections equation 8} and \eqref{subprincipal symbol of projections equation 10} into \eqref{revised algorithm crucial formula} we arrive at
\begin{equation*}
\label{subprincipal symbol of projections equation 11}
\begin{split}
(X_{j,1})_\mathrm{prin}
&
=
\frac i2 \{P^{(j)},P^{(j)}\}-i \,P^{(j)}\{P^{(j)},P^{(j)}\}P^{(j)}
\\
&
+
\frac i2
\sum_{l\ne j}
\frac
{
P^{(j)}\{P^{(j)},A_\mathrm{prin}\}P^{(l)}
-
P^{(l)}\{P^{(j)},A_\mathrm{prin}\}P^{(j)}
}
{
h^{(j)}-h^{(l)}
}
\\
&
+
\frac i2
\sum_{l\ne j}
\frac
{
P^{(l)}\{A_\mathrm{prin},P^{(j)}\}P^{(j)}-P^{(j)}\{A_\mathrm{prin},P^{(j)}\}P^{(l)})
}
{
h^{(j)}-h^{(l)}
}
\\
&
+
\sum_{l\ne j}
\frac
{P^{(l)}A_\mathrm{sub}P^{(j)}+P^{(j)}A_\mathrm{sub}P^{(l)}}
{h^{(j)}-h^{(l)}}\,,
\end{split}
\end{equation*}
which completes the proof of Theorem~\ref{Main results theorem 2}.

\section{A positivity result}
\label{A positivity result}

One of the most useful properties of pseudodifferential projections is that they can be used to construct sign definite operators (modulo $\Psi^{-\infty}$) out of $A$. A rigorous formulation of this statement is provided by Theorem~\ref{Main results theorem 3}, whose proof is given below.

\begin{proof}[Proof of Theorem~\ref{Main results theorem 3}]

For definiteness, will prove the first part of the theorem, namely, formula \eqref{Main results theorem 3 equation 1}. Formula \eqref{Main results theorem 3 equation 2} is proved similarly.

Suppose $j\in\{1,\ldots, m^+\}$ and put
\begin{equation*}
\label{19 July 2020 equation 1}
\widetilde A:=A-2\sum_{k=1}^{m^-}P_{-k}^*AP_{-k}\,.
\end{equation*}
Then
\begin{equation}
\label{19 July 2020 equation 2}
P_j^*AP_j=P_j^*\widetilde AP_j+C,
\end{equation}
where $C\in\Psi^{-\infty}$ is the symmetric operator given by the explicit formula
\begin{equation*}
\label{19 July 2020 equation 3}
C=2\sum_{k=1}^{m^-}P_j^*P_{-k}^*AP_{-k}P_j\,.
\end{equation*}

The operator $\widetilde A$ is elliptic, self-adjoint and semibounded below.
Let $\widetilde\lambda_l$,
\linebreak
$l=1,\dots ,n$, be its negative eigenvalues and $\widetilde\Pi_l$ the corresponding
eigen\-projections.
Put
\begin{equation*}
\label{19 July 2020 equation 4}
\widehat A:=\widetilde A-\sum_{l=1}^n\widetilde\lambda_l\,\widetilde\Pi_l\,.
\end{equation*}
Then
\begin{equation}
\label{19 July 2020 equation 5}
\widehat A\ge0.
\end{equation}

We have
\begin{equation}
\label{19 July 2020 equation 6}
P_j^*\widetilde AP_j=P_j^*\widehat AP_j+D
\end{equation}
where $D\in\Psi^{-\infty}$ is the symmetric operator given by the explicit formula
\begin{equation*}
\label{19 July 2020 equation 7}
D=\sum_{l=1}^n\widetilde\lambda_lP_j^*\widetilde\Pi_lP_j\,.
\end{equation*}

Formula \eqref{19 July 2020 equation 5} implies
\begin{equation}
\label{19 July 2020 equation 8}
P_j^*\widehat AP_j\ge0.
\end{equation}
Combining formulae
\eqref{19 July 2020 equation 2},
\eqref{19 July 2020 equation 6}
and
\eqref{19 July 2020 equation 8}, we get
\begin{equation*}
\label{19 July 2020 equation 9}
P_j^*AP_j\ge C+D\in\Psi^{-\infty}.
\end{equation*}
\end{proof}

Theorem~\ref{Main results theorem 3} will turn out to be quite useful in obtaining spectral-theoretic results in~\cite{part2}.

\section{Modulus and Heaviside function of an elliptic system}
\label{Modulus and Heaviside function of an elliptic system}

In this section we represent the modulus and the Heaviside function of a self-adjoint elliptic matrix pseudodifferential operator $A\in \Psi^s$, $s>0$, in terms of pseudodifferential projections.

It is well-known that $|A|$ and $\theta(A)$ are pseudodifferential operators of order $s$ and $0$, respectively, see, e.g., \cite[Sec.~2]{asymm2}. Seeley's calculus \cite{seeley} allows one, in principle, to compute locally the symbol of $|A|$ and $\theta(A)$ in terms of the symbol of the resolvent $(A-\lambda \,\mathrm{Id})^{-1}$. Carrying out such calculations involves dealing with pseudodifferential operators depending on a parameter, which makes it impractical to push the calculations beyond the very first few terms. In fact, we are unaware of any explicit formulae for $|A|_\mathrm{sub}$ or $[\theta(A)]_\mathrm{sub}$. 

It is worth mentioning that an abstract analysis of the subprincipal symbol of elliptic operators of Laplace and Dirac type acting in vector bundles was performed in \cite{strohmaier}. The Heaviside function of an elliptic system is mentioned in \cite{strohmaier}, though the authors stop short of computing its subprincipal symbol.

Our contribution to the study of the operators $|A|$ and $\theta(A)$ is to establish, via Theorems~\ref{Main results theorem 4} and~\ref{Main results theorem 5}, a relation between such operators and pseudodifferential projections. This yields, in turn, in view of subsection~\ref{The simplified algorithm}, an explicit algorithm for the calculation of the full symbol of $|A|$ and $\theta(A)$, arguably simpler and more straightforward than Seeley's. Note that, in particular, our approach does not involve either complex analysis or pseudodifferential operators depending on a parameter.

\begin{proof}[Proof of Theorem~\ref{Main results theorem 4}]
We have $|A|=\sqrt{A^2}$, so the fact that $|A|$ is a pseudodifferential operator from the class
$\Psi^s$ follows from \cite{seeley}; see also \cite[\S9--\S11]{shubin} for a more detailed exposition.
The basic construction presented in \cite{seeley, shubin} requires the operator $A^2$ to be strictly positive, but
this assumption is not necessarily satisfied for our operator $A$. However, one can deal with this issue as
follows. Suppose that zero is an eigenvalue of $A$ and denote
by
\begin{equation}
\label{eigenprojection onto the kernel of A}
\Pi_0:=\sum_{k:\ \lambda_k=0}\langle v_k,\,\cdot\,\rangle\,v_k
\end{equation}
the eigenprojection onto the kernel of $A$. Then one can define the modulus of $A$ via the identity
$
|A|=\sqrt{A^2+\Pi_0}-\Pi_0
$
which only involves extracting the square root of a strictly positive operator.

In order to prove \eqref{Main results theorem 4 equation 1}, let us argue by contradiction. Suppose there exists a positive integer $k$ such that
\begin{equation}
\label{proof main theorem 4 equation 1}
|A|=\sum_{j=1}^{m^+}AP_j-\sum_{j=1}^{m^-}AP_{-j}-B,
\end{equation}
where 
\begin{equation}
\label{proof main theorem 4 equation 1bis}
B\in \Psi^{s-k}
\quad
\text{but} 
\quad
B\not \in \Psi^{s-k-1}.
\end{equation} 
Then $B_\mathrm{prin}$ has to satisfy
\begin{equation}
\label{proof main theorem 4 equation 2}
|A|_\mathrm{prin} B_\mathrm{prin}+B_\mathrm{prin}|A|_\mathrm{prin}=0,
\end{equation}
because
\begin{equation*}
\left(\sum_{j=1}^{m^+}AP_j-\sum_{j=1}^{m^-}AP_{-j}\right)^2=A^2 \mod \Psi^{-\infty}.
\end{equation*}

Since
$
|A|_\mathrm{prin}=\sum_j |h^{(j)}| P^{(j)},
$
multiplying \eqref{proof main theorem 4 equation 2} by $P^{(n)}$ on the left and by $P^{(l)}$ on the right we obtain
\begin{equation*}
(|h^{(n)}|+|h^{(l)}|)\,P^{(n)} B_\mathrm{prin} P^{(l)}=0,
\end{equation*}
which implies $P^{(n)}B_\mathrm{prin}P^{(l)}=0$ for every $n$ and $l$, i.e.~$B_\mathrm{prin}=0$. But this contradicts~\eqref{proof main theorem 4 equation 1bis}.

Formula \eqref{Main results theorem 4 equation 2} follows from
formula \eqref{Main results theorem 4 equation 1}
and
Theorem~\ref{Main results theorem 2}.
Alternatively, it can be derived from the identity $|A|^2=A^2$.
\end{proof}

\begin{proof}[Proof of Theorem~\ref{Main results theorem 5}]
Let
\begin{equation*}
\label{definition of pseudoinverse of A}
B:=\sum_{k:\, \lambda_k\ne 0} \frac{1}{\lambda_k}\,\langle v_k,\,\cdot\,\rangle\,v_k
\end{equation*}
be the pseudoinverse of $A$ \cite[Chapter 2 Section 2]{rellich}.
Then $AB=BA=\operatorname{Id}-\Pi_0$, where $\Pi_0$ is defined by
\eqref{eigenprojection onto the kernel of A}.
Thus, $B\in\Psi^{-s}$ is a parametrix for $A$.

We have
\[
\theta(A)=\frac{\operatorname{Id}+B|A|-\Pi_0}{2}\,,
\]
and formula \eqref{Main results theorem 5 equation 1} follows now from
Theorem~\ref{Main results theorem 4} and formula
\eqref{definition of system of pseudodifferential projections condition 2}.
\end{proof}

\section{Applications}
\label{Applications}

In this section we discuss some applications of the above results. The most important application --- the partition of the spectrum of a positive order pseudodifferential system --- will be the subject of a separate paper \cite{part2},  where, among other things, results from \cite{CDV, wave, lorentzian, dirac} will be refined and improved.

Throughout this section we adopt Einstein's summation convention over repeated indices.

\subsection{Massless Dirac operator}
\label{Massless Dirac operator}

Let $(M,g)$ be a closed connected orientable and oriented Riemannian 3-manifold. We denote by $\nabla$ the Levi-Civita connection, by $\Gamma^\alpha{}_{\beta\gamma}$ the Christoffel symbols, by $\rho(x):=\sqrt{g_{\alpha\beta}(x)}$ the Riemannian density and by
\begin{equation*}
\label{Pauli matrices basic}
s^1:=
\begin{pmatrix}
0&1\\
1&0
\end{pmatrix}
=s_1
\,,
\quad
s^2:=
\begin{pmatrix}
0&-i\\
i&0
\end{pmatrix}
=s_2
\,,
\quad
s^3:=
\begin{pmatrix}
1&0\\
0&-1
\end{pmatrix}
=s_3
\end{equation*}
the standard Pauli matrices.

Let $e_j$, $j=1,2,3$, be a positively oriented orthonormal global framing. In chosen local coordinates $x^\alpha$, $\alpha=1,2,3$, we denote by $e_j{}^\alpha$ the $\alpha$-th component of the $j$-th vector field. 

The massless Dirac operator $W:H^1(M) \to L^2(M)$ acting on 2-columns of complex-valued half-densities  is the differential operator defined by
\begin{equation}
\label{definition massless Dirac equation}
W:=-i\sigma^\alpha
\left(
\frac\partial{\partial x^\alpha}
+\frac14\sigma_\beta
\left(
\frac{\partial\sigma^\beta}{\partial x^\alpha}
+\Gamma^\beta{}_{\alpha\gamma}\,\sigma^\gamma
\right)
-\frac12 \Gamma^\beta{}_{\alpha\beta}
\right),
\end{equation}
where
\begin{equation*}
\sigma^\alpha(x):=\sum_{j=1}^3 s^j\,e_j{}^\alpha(x).
\end{equation*}

The goal of this subsection is to compute $|W|_\mathrm{sub}$.
Note that $[\theta(W)]_\mathrm{sub}$ was calculated in \cite[subsection 5.2]{dirac}.

Let us begin by observing that a global framing $e_j$, $j=1,2,3$, defines a curvature-free metric compatible affine connection $\nabla^W$ on $M$, known as \emph{Weitzenb\"ock connection}, whose connection coefficients $\Upsilon^\alpha{}_{\beta\gamma}$ read
\begin{equation*}
\label{weitzenbock connection coefficients definition}
\Upsilon^\alpha{}_{\beta\gamma}=- \,e^j{}_\gamma  \dfrac{\partial e_j{}^\alpha}{\partial x^\beta}= e_j{}^\alpha  \dfrac{\partial e^j{}_\gamma}{\partial x^\beta}\,.
\end{equation*}
Here
\begin{equation}
\label{definition of coframe}
e^j{}_\alpha:=\delta^{jk}\,g_{\alpha\beta}\,e_k{}^\beta
\end{equation}
and $\delta^{jk}$ is the Kronecker symbol. We refer the reader to \cite[Appendix~A]{dirac} and references therein for further details.

The \emph{contorsion} of $\nabla^W$ is defined to be the (1,2)-tensor
\begin{equation}
\label{contorsion definition}
K^{\alpha}{}_{\beta\gamma}:=\Upsilon^\alpha{}_{\beta\gamma}-\Gamma^\alpha{}_{\beta\gamma}.
\end{equation}
Contorsion can be expressed in terms of the --- perhaps more familiar --- torsion of the affine connection $\nabla^W$, see~\cite[Equations~(A.3) and~(A.5)]{dirac}, and is equivalent to it. Working with contorsion, as opposed to torsion, is just a matter of convenience.

Lowering the first index in \eqref{contorsion definition} by means of the Riemannian metric $g$, we obtain a (0,3)-tensor antisymmetric in the first and third indices, $K_{\alpha\beta\gamma}=-K_{\gamma\beta\alpha}$. Because in dimension 3 antisymmetric tensors of order two are equivalent to vectors, instead of working with contorsion \eqref{contorsion definition} we can equivalently work with
\begin{equation}
\overset{*}{K}_{\alpha\beta}:=\frac12 K^\mu{}_\alpha{}^\nu \,E_{\mu\nu\beta},
\end{equation}
where 
\begin{equation}
\label{definition of E}
E_{\alpha\beta\gamma}(x):=\rho(x)\, \varepsilon_{\alpha\beta\gamma}
\end{equation}
and $\varepsilon$ is the totally antisymmetric symbol, $\varepsilon_{123}=+1$.

From \eqref{definition massless Dirac equation} it is easy to see that
\begin{equation}
\label{principal symbol dirac}
W_\mathrm{prin}(x,\xi)=\sigma^\alpha \xi_\alpha.
\end{equation}
By direct computation one can establish that the eigenvalues of $W_\mathrm{prin}$ are simple and read $h^{(\pm 1)}=\pm h$, where
\begin{equation}
\label{hamiltonian}
h(x,\xi):=\sqrt{g^{\alpha\beta}(x) \xi_\alpha\xi_\beta}.
\end{equation}

Combining \cite[Lemma~6.1]{jst_part_b} with identities from \cite[Appendix~A]{dirac} one can show that
\begin{equation}
\label{subprincipal symbol dirac}
W_\mathrm{sub}(x)=-\frac12 \overset{*}{K}{}^\alpha{}_\alpha(x)\,I\,.
\end{equation}

\begin{theorem}
\label{theorem subprincipal symbol |W|}
The subprincipal symbol of the modulus of the massless Dirac operator $|W|$ is given by
\begin{equation}
\label{theorem subprincipal symbol |W| equation 1}
|W|_\mathrm{sub}=-\frac 1{2h}\overset{*}{K}{}^\alpha{}_\beta \,\xi_\alpha\, (W_\mathrm{prin})_{\xi_\beta}.
\end{equation}
\end{theorem}

\begin{proof}
When $A$ is chosen to be the massless Dirac operator $W$, formula \eqref{Main results theorem 4 equation 1} from Theorem~\ref{Main results theorem 4} can be simplified to read
\begin{multline}
\label{proof theorem subprincipal symbol |W| equation 1}
|W|_\mathrm{sub}
=
\sum_{j=\pm1}
\operatorname{sgn}(h^{(j)})
P^{(j)}W_\mathrm{sub}P^{(j)}
+
\frac i{4h}
\bigl(
\{W_\mathrm{prin},W_\mathrm{prin}\}
-
\{|W|_\mathrm{prin},|W|_\mathrm{prin}\}
\bigr).
\end{multline}
The task at hand is to compute the RHS of \eqref{proof theorem subprincipal symbol |W| equation 1}.

In view of \eqref{subprincipal symbol dirac}, the first term on the RHS of \eqref{proof theorem subprincipal symbol |W| equation 1} can be evaluated as
\begin{equation}
\label{proof theorem subprincipal symbol |W| equation 2}
\begin{split}
\sum_{j=\pm1}
\operatorname{sgn}(h^{(j)})
P^{(j)}W_\mathrm{sub}P^{(j)}
&
= 
-\frac12 \overset{*}{K}{}^\alpha{}_\alpha \,(P^{(1)}-P^{(-1)})
\\
&
=
-\frac1{2h} \overset{*}{K}{}^\alpha{}_\alpha\, W_\mathrm{prin}.
\end{split}
\end{equation}

Direct calculations involving \eqref{principal symbol dirac} and \eqref{poisson brackets} give us
\begin{equation}
\label{proof theorem subprincipal symbol |W| equation 3}
\{W_\mathrm{prin},W_\mathrm{prin}\}
=
[\sigma^\alpha_{x^\mu} , \sigma^\mu]\,\xi_\alpha.
\end{equation}
Let us choose a point $y\in M$. Let $\{\widetilde e_j\}_{j=1}^3$ be the Levi-Civita framing generated by $\{e_j\}_{j=1}^3$ at $y$ defined in accordance with \cite[Definition~7.1]{dirac}, and let $\widetilde{W}$ be the massless Dirac operator associated with the latter. Then, formula \eqref{proof theorem subprincipal symbol |W| equation 3} and \cite[Corollary~7.3]{dirac} imply
\begin{equation}
\label{proof theorem subprincipal symbol |W| equation 4}
\{\widetilde W_\mathrm{prin},\widetilde W_\mathrm{prin}\}=0.
\end{equation}
Now, there exists a smooth matrix-function $G:M \to SU(2)$ such that the framings $\{e_j\}_{j=1}^3$ and $\{\widetilde e_j\}_{j=1}^3$ are related in accordance with
\begin{equation*}
\label{proof theorem subprincipal symbol |W| equation 5}
{e}_j{}^\alpha(x)=\frac12 \operatorname{tr}(s_j\,G^*(x)\,s^k\,G(x))\,\widetilde{e}_k{}^\alpha(x), \qquad G(y)=\mathrm{Id}.
\end{equation*}
It is easy to see that the corresponding Dirac operators and their principal symbols satisfy
\begin{equation}
\label{proof theorem subprincipal symbol |W| equation 6}
W=G^* \,\widetilde{W} \, G \qquad \text{and}\qquad W_\mathrm{prin}=G^* \, \widetilde W_\mathrm{prin} \,G,
\end{equation}
respectively.

Substituting \eqref{proof theorem subprincipal symbol |W| equation 6} into 
\eqref{proof theorem subprincipal symbol |W| equation 3} and using \eqref{proof theorem subprincipal symbol |W| equation 4} we obtain
\begin{equation}
\label{proof theorem subprincipal symbol |W| equation 7}
\begin{split}
\{W_\mathrm{prin},W_\mathrm{prin}\}(y)
&
=
\left. [G^*_{x^\alpha}\widetilde{W}_\mathrm{prin}+\widetilde{W}_\mathrm{prin}G_{x^\alpha}, \sigma^\alpha]\right|_{x=y}.
\end{split}
\end{equation}
Formula \cite[Eqn.~(7.53)]{dirac} tells us that 
\begin{equation}
\label{proof theorem subprincipal symbol |W| equation 8}
G_{x^\alpha}(y)=-\frac i2 \overset{*}{K}{}_{\alpha\beta}(y) \,\sigma^\beta(y), 
\qquad
G_{x^\alpha}^*(y)=\frac i2 \overset{*}{K}{}_{\alpha\beta}(y) \,\sigma^\beta(y).
\end{equation}
Substituting \eqref{proof theorem subprincipal symbol |W| equation 8} into \eqref{proof theorem subprincipal symbol |W| equation 7} and using elementary properties of Pauli matrices we get
\begin{equation}
\label{proof theorem subprincipal symbol |W| equation 9}
\{W_\mathrm{prin},W_\mathrm{prin}\}(y)
=
-2i \left. \bigl(\overset{*}{K}{}^\alpha{}_\alpha \,W_\mathrm{prin} - \overset{*}{K}{}^\alpha{}_\beta \,\xi_\alpha\, \sigma^\beta\bigr)\right|_{x=y}. 
\end{equation}

We are left with computing $\{|W|_\mathrm{sub},|W|_\mathrm{sub}\}$. It follows from formula \eqref{principal symbol dirac} that
\begin{equation*}
\label{proof theorem subprincipal symbol |W| equation 10}
|W|_\mathrm{prin}=|W_\mathrm{prin}|=h\,I\,,
\end{equation*}
which, in turn, immediately implies
\begin{equation}
\label{proof theorem subprincipal symbol |W| equation 11}
\{|W|_\mathrm{prin},|W|_\mathrm{prin}\}=\{h,h\}\,I=0.
\end{equation}

After observing that the point $y$ chosen above is arbitrary, it only remains to substitute \eqref{proof theorem subprincipal symbol |W| equation 2}, \eqref{proof theorem subprincipal symbol |W| equation 9} and \eqref{proof theorem subprincipal symbol |W| equation 11} into \eqref{proof theorem subprincipal symbol |W| equation 1}.
\end{proof}

It is instructive to compare formulae
\eqref{theorem subprincipal symbol |W| equation 1}
and
\eqref{subprincipal symbol dirac}:
we see that $|W|_\mathrm{sub}$ is trace-free
whereas $W_\mathrm{sub}$ is pure trace, i.e.~proportional to the identity matrix.


\

Let us specialise further to the case of the 3-sphere, 
\[
M=\mathbb{S}^3:=\{\mathbf{x}\in \mathbb{R}^4 \,|\, \|\mathbf{x}\|=1\},
\] 
equipped with the standard round metric $g_{\mathbb{S}^3}$, with orientation prescribed in accordance with \cite[Appendix~A]{sphere}. If we choose our positively oriented framing $e_j$, $j=1,2,3$, to be the restriction to $\mathbb{S}^3$ of the vector fields in $\mathbb{R}^4$
\begin{equation*}
\label{31 January 2019 formula 8}
\begin{aligned}
&\mathbf{e}_{1}:=-\mathbf{x}^4 \dfrac{\partial}{\partial \mathbf{x}^1}- \mathbf{x}^3\frac{\partial}{\partial \mathbf{x}^2} + \mathbf{x}^2\dfrac{\partial}{\partial \mathbf{x}^3}+\mathbf{x}^1\dfrac{\partial}{\partial \mathbf{x}^4},\\
&\mathbf{e}_{2}:=
+\mathbf{x}^3\dfrac{\partial}{\partial \mathbf{x}^1}-\mathbf{x}^4 \frac{\partial}{\partial \mathbf{x}^2}- \mathbf{x}^1 \dfrac{\partial}{\partial \mathbf{x}^3}+\mathbf{x}^2\dfrac{\partial}{\partial \mathbf{x}^4},\\
&\mathbf{e}_{3}:=
- \mathbf{x}^2\dfrac{\partial}{\partial \mathbf{x}^1}+ \mathbf{x}^1\frac{\partial}{\partial \mathbf{x}^2} -\mathbf{x}^4 \dfrac{\partial}{\partial \mathbf{x}^3}+\mathbf{x}^3\dfrac{\partial}{\partial \mathbf{x}^4},
\end{aligned}
\end{equation*}
then we obtain
\begin{equation*}
\label{einstein framing plus}
\overset{*}{K}=-g_{\mathbb{S}^3}\,.
\end{equation*}
Therefore, Theorem~\ref{theorem subprincipal symbol |W|} yields
\begin{equation}
\label{subprincipal modulus W on 3-sphere}
|W_{\mathbb{S}^3}|_\mathrm{sub}=\frac{1}{2h} (W_{\mathbb{S}^3})_\mathrm{prin}\,.
\end{equation}

\subsection{Elasticity operator}
\label{Elasticity operator}

Let $(M,g)$ be a Riemannian manifold without boundary.
Consider a diffeomorphism $\varphi:M\to M$ which is sufficiently close to the identity, so that it can be represented in
terms of a vector field of displacements $v$, see \cite[formula (4.1)]{diffeo}.
Let $h:=\varphi^*g$ be the pullback of $g$ via $\varphi$. We define the \emph{strain tensor}
to be
\begin{equation}
\label{definition of the strain tensor}
\mathcal{S}^\alpha{}_\beta:=
\frac12(g^{\alpha\gamma}h_{\gamma\beta}-\delta^\alpha{}_\beta).
\end{equation}
Note that in \cite{diffeo} the factor $\tfrac12$ was dropped for the sake of convenience,
but in the current paper we stick with the more traditional definition \eqref{definition of the strain tensor},
see also \cite[formulae (1.2) and (1.3)]{MR0106584}.

Following \cite{diffeo}, let us introduce the scalar invariants $e_k(\varphi)$, $k=1,\dots,d$,
where $e_k(\varphi)$ is the $k$-th elementary symmetric polynomial in the eigenvalues
of $\mathcal{S}$ viewed as a linear operator in the tangent fibre.
In particular, we have
\begin{align}
\label{scalar invariant 1}
e_1(\varphi)&:=
\operatorname{tr}\mathcal{S},
\\
\label{scalar invariant 2}
e_2(\varphi)&:=
\frac12
\left[
(\operatorname{tr}\mathcal{S})^2-\operatorname{tr}(\mathcal{S}^2)
\right],
\end{align}
where $\,\operatorname{tr}\,$ is the matrix trace.

The theory of linear elasticity is based on the following two assumptions.
\begin{enumerate}[(i)]
\item
The potential energy of elastic deformation is quadratic (homogeneous of degree two) in strain.
\item
The strain tensor has been linearised in displacements $v$.
\end{enumerate}

Under the above assumptions the potential energy of elastic deformation reads
\begin{equation}
\label{assumption 1 formula}
E(v)
=
\int_M
(a(e_1)^2+be_2)\,\rho\,dx
\end{equation}
and formula \eqref{definition of the strain tensor} becomes
\begin{equation}
\label{assumption 2 formula}
\mathcal{S}^\alpha{}_\beta=\frac12(\nabla^\alpha v_\beta+\nabla_\beta v^\alpha).
\end{equation}
The operator of linear elasticity $L$, acting on vector fields,  is defined via
\begin{equation}
\label{definition of operator of linear elasticity}
\frac12\int_M
v^\alpha(Lv)^\beta g_{\alpha\beta}\,\rho\,dx
=
E(v)
\end{equation}
and
\eqref{scalar invariant 1}--\eqref{assumption 2 formula}.
Here $\rho$ is the Riemannian density, $\nabla$ is the Levi-Civita connection associated with $g$
and $a,b\in\mathbb{R}$ are parameters.

The tradition in elasticity theory is to express the parameters $a$ and $b$
in terms of the so-called \emph{Lam\'e parameters} $\lambda$ and $\mu$,
\begin{equation}
\label{Lame parameters}
a=\tfrac12\lambda+\mu,\quad b=-2\mu,
\end{equation}
which are assumed to satisfy conditions
\begin{equation}
\label{conditions on Lame parameters}
\mu>0,\qquad\lambda+\frac2d\,\mu>0,
\end{equation}
that guarantee strong convexity,
see, for example, \cite{rozenblum}.
Formula
\eqref{assumption 1 formula}
takes now the more familiar form
\begin{multline}
\label{potential energy in terms of Lame parameters}
E(v)
=
\int_M
\bigl(
\tfrac12\lambda(\operatorname{tr}\mathcal{S})^2+\mu\operatorname{tr}(\mathcal{S}^2)
\bigr)\,\rho\,dx
\\
=
\frac12
\int_M
\bigl(
\lambda(\nabla_\alpha v^\alpha)^2
+\mu
(\nabla_\alpha v_\beta+\nabla_\beta v_\alpha)\nabla^\alpha v^\beta
\bigr)\,\rho\,dx\,,
\end{multline}
see also \cite[formula (4.1)]{MR0106584}.
Substituting \eqref{potential energy in terms of Lame parameters}
into \eqref{definition of operator of linear elasticity}
and integrating by parts we arrive at the explicit formula for the
operator of linear elasticity (Lam\'e operator)
\begin{equation}
\label{operator of linear elasticity explicit formula}
(Lv)^\alpha
=
-\mu(\nabla_\beta\nabla^\beta v^\alpha+\operatorname{Ric}^\alpha{}_\beta v^\beta)
-(\lambda+\mu)\nabla^\alpha\nabla_\beta v^\beta,
\end{equation}
where $\operatorname{Ric}$ is Ricci curvature.

The eigenvalues of $L_\mathrm{prin}$ are as follows:
simple eigenvalue $(\lambda+2\mu)h^2$
and eigenvalue
$\mu h^2$ of multiplicity $d-1$,
where $h$ is defined by \eqref{hamiltonian}.
These correspond to longitudinal and transverse waves, respectively.
Our method requires the eigenvalues of the principal symbol to be simple,
see Assumption~\ref{assumption about eigenvalues of principal symbol being simple},
so further on in this subsection we restrict ourselves to the case $d=2$.

The operator $L$ does not fit into our scheme because it acts on 2-vectors as opposed to
2-columns of half-densities. In order to address this issue, we recast it as follows.

We first switch from 2-vectors to 2-columns of scalar functions by projecting onto a framing,
which requires the manifold $M$ to be parallelizable.
Poincar\'e's theorem
\cite{Milnor book} tells us that a closed connected 2-manifold is parallelizable
if and only if it is, topologically, a 2-torus, so further on in this subsection we restrict ourselves
to the case $M=\mathbb{T}^2$.

Let $e_j$, $j=1,2$, be a positively oriented orthonormal global framing on $\mathbb{T}^2$. In chosen local coordinates $x^\alpha$, $\alpha=1,2$, we denote by $e_j{}^\alpha$ the $\alpha$-th component of the $j$-th vector field.
Then the operator
\begin{equation*}
\label{scalars to vectors}
(Bu)^\alpha:=e_j{}^\alpha u^j
\end{equation*}
maps scalars to vectors and the operator
\begin{equation*}
\label{vectors to scalars}
(B^{-1}v)^j:=e^j{}_\alpha v^\alpha
\end{equation*}
maps vectors to scalars, see~\eqref{definition of coframe}.
We define the operator $L_\mathrm{scal}$ acting on 2-columns of scalar functions as
\begin{equation*}
\label{operator of linear elasticity explicit formula scalar}
L_\mathrm{scal}:=B^{-1}LB.
\end{equation*}
Finally, we turn $L_\mathrm{scal}$ into the operator
\begin{equation*}
\label{operator of linear elasticity explicit formula scalar}
L_{1/2}:=\rho^{1/2}L_\mathrm{scal}\,\rho^{-1/2}
\end{equation*}
acting on 2-columns of half-densities. Here $\rho$ is the Riemannian density.

The operator $L_{1/2}$ is now of the type considered in this paper, so we can apply to it our results.

We have
\begin{equation}
\label{operator of linear elasticity explicit formula scalar principal}
(L_{1/2})_\mathrm{prin}=
\mu h^2I
+
(\lambda+\mu)h^2pp^T,
\end{equation}
where
\begin{equation}
\label{column w}
p:=h^{-1}
\begin{pmatrix}
e_1{}^\alpha\xi_\alpha
\\
e_2{}^\alpha\xi_\alpha
\end{pmatrix}.
\end{equation}
The eigenvalues of $(L_{1/2})_\mathrm{prin}$ are
$h^{(1)}=\mu h^2$
and
$h^{(2)}=(\lambda+2\mu)h^2$,
and, in view of \eqref{conditions on Lame parameters},
these eigenvalues are simple.
Here $h$ is defined by formula \eqref{hamiltonian}.
The corresponding eigenprojections are
$P^{(1)}=I-pp^T$
and
$P^{(2)}=pp^T$.
Note that formula \eqref{conditions on Lame parameters} implies
\begin{equation}
\label{ratio of two eigenvalues of principal symbol}
h^{(2)}/h^{(1)}>2(1-d^{-1}).
\end{equation}

The subprincipal symbol of the operator $L_{1/2}$ is expressed via the torsion tensor
\begin{equation*}
\label{torsion definition}
T^\alpha{}_{\beta \gamma}= \Upsilon^\alpha{}_{\beta\gamma}-\Upsilon^\alpha{}_{\gamma \beta}
\end{equation*}
of the Weitzenb\"ock connection associated with our framing, see previous subsection for details.
In dimension two the torsion tensor is equivalent to a covector field
$t_\alpha:=\tfrac12T_\alpha{}^{\beta \gamma}E_{\beta\gamma}$, where
$E_{\alpha\beta}(x):=\rho(x)\, \varepsilon_{\alpha\beta}$, $\varepsilon_{12}=+1$,
compare with \eqref{definition of E}.
It is easy to see that $t$ is a closed 1-form: the fact that the exterior derivative of $t$ is zero is a consequence
of the fact that the Weitzenb\"ock connection is flat. Namely, let
$
\mathfrak{R}^\alpha{}_{\beta\gamma\delta}
=
\partial_\gamma\Upsilon^\alpha{}_{\delta\beta}
-
\partial_\delta\Upsilon^\alpha{}_{\gamma\beta}
+
\Upsilon^\alpha{}_{\gamma\kappa}
\Upsilon^\kappa{}_{\delta\beta}
-
\Upsilon^\alpha{}_{\delta\kappa}
\Upsilon^\kappa{}_{\gamma\beta}
$
be the curvature tensor of the Weitzenb\"ock connection (which is zero by definition).
Then $\mathfrak{R}$ and $dt$ are related as
$\mathfrak{R}_{\alpha\beta\gamma\delta}
=
E_{\alpha\beta}(dt)_{\gamma\delta}
=
(dt)_{\alpha\beta}E_{\gamma\delta}$.

Straightforward but lengthy calculations give us
\begin{equation}
\label{operator of linear elasticity explicit formula scalar subprincipal}
(L_{1/2})_\mathrm{sub}=
i(\lambda+3\mu)\,t^\alpha\xi_\alpha\,\epsilon\,,
\end{equation}
where
$
\epsilon:=
\begin{pmatrix}
0&1\\
-1&0
\end{pmatrix}
$.
The matrix $\epsilon$
admits a simple geometric interpretation.
Namely, let $u=\begin{pmatrix}
u^1\\
u^2
\end{pmatrix}\in\mathbb{R}^2$
be a 2-column of scalars.
Then $u\mapsto\epsilon u$ is the operator of rotation by $\pi/2$,  clockwise,  in the $\mathbb{R}^2$ plane.

One can show that
\begin{equation*}
\label{elasticity explanation 1}
\{P^{(j)},P^{(k)}\}=0,\qquad j,k=1,2,
\end{equation*}
\begin{equation*}
\label{elasticity explanation 2}
P^{(1)}\epsilon P^{(2)}
+
P^{(2)}\epsilon P^{(1)}
=
\epsilon\,,
\end{equation*}
\begin{equation*}
\label{elasticity explanation 3}
P^{(2)}Q^{(j)}P^{(1)}
-
P^{(1)}Q^{(j)}P^{(2)}
=
(-1)^j
(\lambda+3\mu)\,t^\alpha\xi_\alpha\,\epsilon\,,
\qquad
j=1,2,
\end{equation*}
so that Theorem \ref{Main results theorem 2} gives us
\begin{equation}
\label{operator of linear elasticity explicit formula scalar subprincipal projections}
(P_1)_\mathrm{sub}
=
(P_2)_\mathrm{sub}
=
0.
\end{equation}

There is an underlying reason for the subprincipal symbols of our pseudodifferential projections being zero
in this particular case, i.e.~for linear elasticity in dimension two.
In dimension two a 1-form $v$ can be written, locally, in terms of two scalar potentials
$\varphi$ and $\psi$ as
\begin{equation}
\label{1-form in terms of potentials}
v=d\varphi+*d\psi,
\end{equation}
where $*$ is the Hodge dual (rotation by $\pi/2$ in the cotangent fibre).
The operator of linear elasticity $L$ agrees well with the substitution \eqref{1-form in terms of potentials}
in that it decouples $L\,$, modulo lower order terms involving curvature,
into a pair of scalar operators acting on $\varphi$ and $\psi$ separately.
In the scalar case ($m=1$) the pseudodifferential projection from Theorem~\ref{Main results theorem 1}
would simply be the identity operator whose subprincipal symbol is, obviously, zero.

\subsection{The Dirichlet-to-Neumann map of linear elasticity}

Let $\Omega\subset\mathbb{R}^3$ be a bounded domain (connected open set) with smooth boundary $M$.
One could consider the more general case of a Riemannian 3-manifold with boundary but we refrain
from doing this in the current paper for the sake of simplicity.


We denote Cartesian coordinates in $\mathbb{R}^3$ by $y=(y^1,y^2,y^3)$
and local coordinates on $M$ by $x=(x^1,x^2)$. By $x^3$ we shall denote
the signed distance from a point $P\in\mathbb{R}^3$ to $M$,
positive for $P\not\in\overline{\Omega}$ and negative for $P\in\Omega$.
Clearly, $(x^1,x^2,x^3)$ are local coordinates in $\mathbb{R}^3$.
We assume that the orientation of $(x^1,x^2,x^3)$ agrees with that of $y=(y^1,y^2,y^3)$,
and this defines an orientation of $M$.

Consider the variational functional of 3-dimensional linear elasticity
\begin{equation}
\label{potential energy in terms of Lame parameters 3d}
E(v)
=
\frac12
\int_\Omega
\bigl(
\lambda(\partial_\alpha v^\alpha)^2
+\mu
(\partial_\alpha v_\beta+\partial_\beta v_\alpha)\partial^\alpha v^\beta
\bigr)\,dy\,.
\end{equation}
Integration by parts gives
\begin{equation}
\label{definition of traction}
E(v)
=
\frac12\int_\Omega
v_\alpha(Lv)^\alpha\,dy\,
+
\frac12\int_M
v_\alpha(Tv)^\alpha\,dS\,,
\end{equation}
where $L$ is the operator of linear elasticity (see previous subsection)
and $T$ is \emph{traction}.
Traction $T$ is a first order linear partial differential operator mapping
a 3-dimensional vector field in $\Omega$ to a 3-dimensional vector field defined on $M$.

Let $w$ be a 3-dimensional vector field defined on $M$.
It is known that the Dirichlet problem for the elasticity operator
\begin{equation}
\label{Dirichlet problem for the elasticity operator 1}
Lv=0,
\end{equation}
\begin{equation}
\label{Dirichlet problem for the elasticity operator 2}
v|_M=w
\end{equation}
has a unique solution.
Let $L_D^{-1}$ be the linear operator mapping $w$ to $v$, a solution to
\eqref{Dirichlet problem for the elasticity operator 1},
\eqref{Dirichlet problem for the elasticity operator 2}.

The Dirichlet-to-Neumann map of linear elasticity is the linear operator
\begin{equation}
\label{Dirichlet-to-Neumann map of linear elasticity definition}
DN:=TL_D^{-1}.
\end{equation}
This operator acts in the linear space of
3-dimensional vector fields defined on $M$.

Further on when working with 3-dimensional vector fields defined on $M$ we will use local coordinates
$(x^1,x^2,x^3)$, so that our vector fields have the structure $v=(v^1,v^2,v^3)$,
where $(v^1,v^2)$ is a vector field on $M$ and $v^3$ is the normal component of $v$
which can be viewed as a scalar function. The inner product on such vector fields is defined as
\begin{equation}
\label{Dirichlet-to-Neumann inner product}
(v,w):=\int_M(g_{\alpha\beta}v^\alpha w^\beta+v^3w^3)\,dS
\end{equation}
where $g$ is the Riemannian metric on $M$ induced by the Euclidean metric on $\mathbb{R}^3$
and summation is carried out over $\alpha,\beta=1,2$.

The operator $DN$ is a first order $3\times3$ matrix pseudodifferential operator,
symmetric with respect to the inner product
\eqref{Dirichlet-to-Neumann inner product}.
Furthermore, it is self-adjoint as an operator from $H^1(M)$ to $L^2(M)$,
elliptic, nonnegative and has a 6-dimensional kernel
(rigid translations and rotations).

Let $\nu=\frac\lambda{2(\lambda+\mu)}$ be Poisson's ratio.
The eigenvalues of the principal symbol of the operator $DN$,
enumerated in increasing order $0<h^{(1)}\le h^{(2)}<h^{(3)}$, are expressed via  Poisson's ratio as follows:
if $-1<\nu<1/4$ then
\begin{equation}
\label{Dirichlet-to-Neumann eigenvalues of the principal symbol for nu less than one quarter}
h^{(1)}
=
\frac{2\mu h}{3-4\nu}\,,
\qquad
h^{(2)}
=
\mu h\,,
\qquad
h^{(3)}
=
2\mu h\,,
\end{equation}
and if $1/4<\nu<1/2$ then
\begin{equation}
\label{Dirichlet-to-Neumann eigenvalues of the principal symbol for nu greater than one quarter}
h^{(1)}
=
\mu h\,,
\qquad
h^{(2)}
=
\frac{2\mu h}{3-4\nu}\,,
\qquad
h^{(3)}
=
2\mu h\,.
\end{equation}
For $\nu=1/4$ we get a double eigenvalue $h^{(1)}=h^{(2)}=\mu h$.
Here $h$ is defined by formula \eqref{hamiltonian},
with summation carried out over $\alpha,\beta=1,2$.

Formulae
\eqref{Dirichlet-to-Neumann eigenvalues of the principal symbol for nu less than one quarter}
and
\eqref{Dirichlet-to-Neumann eigenvalues of the principal symbol for nu greater than one quarter}
can be derived by considering an elastic half-space $x^3<0$ and performing separation of variables.
Alternatively, one can compute the principal symbol of the operator $DN$ and its eigenvalues by expressing
this operator in terms of single and double layer potentials, see
\cite{AgranovichLevitin}.
Note that for all $-1<\nu<1/2$ we have $2\le h^{(3)}/h^{(1)}<7$,
compare with~\eqref{ratio of two eigenvalues of principal symbol}.

As explained in the previous subsection, in order to construct our pseudodifferential projections we need
eigenvalues of the principal symbol to be simple
and $M$ to be parallelizable, which means that our construction would work when $\nu\ne1/4$
and $M$ is, topologically, a disjoint union of 2-tori. Admissible examples of 
3-dimensional domains $\Omega$ would be, say,
a doughnut or a toroidal shell.

Performing a detailed construction of pseudodifferential projections for the operator $DN$
is a lengthy and technical enterprise which
would shift the focus of this paper away from its core topic.
We, therefore, decided not to carry out a full analysis of this meaningful example in the current paper
and plan to address this matter comprehensively elsewhere.

\section*{Acknowledgements}
\addcontentsline{toc}{section}{Acknowledgements}

We are grateful to Alex Strohmaier for interesting comments at an early stage of this project, for bringing the reference \cite{bolte} to our attention and for stimulating discussions. 
We would also like to thank Nikolai Saveliev and Grigori Rozenbloum for useful bibliographic suggestions, as well as Daniel Grieser and an anonymous referee for insightful comments.

MC was supported by a Leverhulme Trust Research Project Grant RPG-2019-240 and by a Research Grant (Scheme 4) of the London Mathematical Society. Both are gratefully acknowledged.

%
%
%


\begin{thebibliography}{41}
\addcontentsline{toc}{section}{References}

\bibitem{AgranovichLevitin}
M.~S.~Agranovich, B.~A.~Amosov and M.~Levitin,
Spectral problems for the Lam\'e system with spectral parameter in boundary conditions
on smooth or nonsmooth boundary,
{\it Russ. J. Math. Phys.} \textbf{6} no.~3 (1999) 247--281.


\bibitem{asymm2}
M.~F.~Atiyah, V.~K.~Patodi and I.~M.~Singer, 
Spectral asymmetry and Riemannian geometry I,
{\it Math. Proc. Camb. Phil. Soc.} \textbf{77} (1975) 43--69.

%
%


%


%
%
%
%

\bibitem{baum}
P.~Baum and R.~G.~Douglas,
Toeplitz operators and Poincar\'e duality, 
in: {\it Toeplitz Centennial}, I.~Gohberg (Ed.), {\it Operator Theory: Advances and Applications} {\bf 4} Birkh\"auser, Basel (1981) 137--166.

\bibitem{birman}
M.~Sh.~Birman and M.~Z.~Solomyak,
On subspaces that admit a pseudodifferential projector, (Russian) {\it Vestnik Leningrad. Univ. Mat. Mekh. Astronom.} \textbf{82} no.~1 (1982) 18--25. English translation: {\it Vestnik Leningrad. Univ. Mat. Mekh. Astronom.} \textbf{15} (1982) 17--27. 

%

\bibitem{bolte}
J.~Bolte and R.~Glaser, 
Semiclassical Egorov theorem and quantum ergodicity for matrix valued operators,
{\it Comm.~Math.~ Phys. } {\bf 247} (2004) 391--419.

\bibitem{woj3}
B.~Booss-Bavnbek and K.~P.~Wojciechowski,
Desuspension of splitting elliptic symbols~I, 
{\it Ann.~Global Anal.~Geom.} \textbf{3} no.~3 (1985) 337--383. 

\bibitem{woj4}
B.~Booss-Bavnbek and K.~P.~Wojciechowski,
Desuspension of splitting elliptic symbols~II, 
{\it Ann.~Global Anal.~Geom.} \textbf{4} no.~3 (1986) 349--400. 

\bibitem{woj5}
B.~Booss-Bavnbek and K.~P.~Wojciechowski,
Pseudo-differential projections and the topology of certain spaces of elliptic boundary value problems,
{\it Comm.~Math.~Phys.} \textbf{121} no.~1 (1989) 1--9. 

%
%

\bibitem{BR99}
V.~Bruneau and D.~Robert,
Asymptotics of the scattering phase for the Dirac operator: High energy, semi-classical and non-relativistic limits,
{\it Ark.  Mat. } {\bf 37} (1999) 1--32.

\bibitem{calderon}
A.~Calder\'on, 
Boundary value problems for elliptic equations, 
Outlines for the Joint Soviet-American Symposium on Partial Differential Equations,  Novosibirsk (1963) 303--304.

\bibitem{diagonalisation}
M.~Capoferri,
Diagonalization of elliptic systems via pseudodifferential projections,
{\it J. Differential Equations} {\bf 313} (2022) 157--187.


\bibitem{lorentzian}
M.~Capoferri, C.~Dappiaggi and N.~Drago,
Global wave parametrices on globally hyperbolic spacetimes,
{\it J. Math. Anal. Appl.} {\bf 490} (2020) 124316.

\bibitem{wave}
M.~Capoferri, M.~Levitin and D.~Vassiliev,
Geometric wave propagator on Riemannian manifolds.
Preprint arXiv:1902.06982 (2019), to appear in {\it Comm. Anal. Geom.}

\bibitem{diffeo}
M.~Capoferri and D.~Vassiliev,
Spacetime diffeomorphisms as matter fields,
{\it J. Math. Phys.} {\bf 61} (2020) 111508.


\bibitem{dirac}
M.~Capoferri and D.~Vassiliev,
Global propagator for the massless Dirac operator and spectral asymptotics,
Preprint arXiv:2004.06351 (2020).

\bibitem{part2}
M.~Capoferri and D.~Vassiliev,
Invariant subspaces of elliptic systems II: spectral theory,
{\it J.~Spectr.~Theory}, to appear.

\bibitem{CDV}
O.~Chervova, R.~J.~Downes and D.~Vassiliev,
The spectral function of a first order elliptic system,
{\it J. Spectr. Theory} \textbf{3} no.~3 (2013) 317--360. 

\bibitem{jst_part_b}
O.~Chervova, R.~J.~Downes and D.~Vassiliev,
Spectral theoretic characterization of the massless Dirac operator,
{\it J.~London Math.~Soc.} \textbf{89} (2014) 301--320.

\bibitem{cordes3}
H.~O.~Cordes, 
A version of Egorov’s theorem for systems of hyperbolic pseudo-differential equations,
{\it J. Funct. Anal.} {\bf 48} no.~3 (1982) 285--300.

\bibitem{cordes}
H.~O.~Cordes, 
A pseudodifferential-Foldy-Wouthuysen transform,
{\it Comm.  Partial Differential Equations} {\bf 8} (1983) 1475--1485.


\bibitem{DuGu}
J.~J.~Duistermaat and V.~W.~Guillemin,
The spectrum of positive elliptic operators and periodic bicharacteristics,
{\it Invent. Math.} \textbf{29}  no.~1 (1975) 39--79.


\bibitem{LF91}
R.~G.~Littlejohn and W.~G.~Flynn,  
Geometric phases in the asymptotic theory of coupled wave equations,
{\it Phys.  Rev.  A} {\bf 44} (1991) 5239--5256.

\bibitem{sphere}
Y.-L.~Fang, M.~Levitin and D.~Vassiliev,
Spectral analysis of the Dirac operator on a 3-sphere,
{\it Operators and Matrices} \textbf{12} (2018) 501--527.

%

\bibitem{grieser}
K.~Fritzsch, D.~Grieser and E.~Schrohe,
The Calder\'on Projector for Fibred Cusp Operators.
Preprint arXiv:2006.04645 (2020).

%
 
\bibitem{gohberg}
I.~Gohberg and M.~G.~Krein, 
Systems of integral equations on the half-line with kernels depending on the difference of the arguments, 
{\it Uspekhi Matem. Nauk} {\bf 13} no.~2 (1958) 3--72.

%
%
%
\bibitem{Hor}
L.~H\"ormander,
{\it The analysis of linear partial differential operators. I}.
Reprint of the second (1990) edition. Classics in Mathematics. Springer-Verlag, Berlin, 2003;
{\it III}. Reprint of the 1994 edition. Classics in Mathematics. Springer-Verlag, Berlin, 2007; {\it IV}. Reprint of the 1994 edition. Classics in Mathematics. Springer-Verlag, Berlin, 2009.
%
%
%


%
%

\bibitem{MR0106584}
L.~D.~Landau and E.~M.~Lifshitz,
{\it Theory of elasticity, course of theoretical physics vol~7}, 3rd edn
(Pergamon, Oxford, 1986).
Translated from the Russian by J.~B.~Sykes and W.~H.~Reid.

\bibitem{Milnor book}
J.~Milnor, {\it Topology from the differentiable viewpoint}.
University Press of Virginia, 1965.

%
%

\bibitem{rellich}
F.~Rellich,
\emph{Perturbation theory of eigenvalue problems}.
Courant Institute of Mathematical Sciences, New York University, 1954.

\bibitem{seeley1}
R.~Seeley, 
Singular Integrals and Boundary Value Problems, 
{\it Amer. J. Math.} {\bf 88} no.~4 (1966) 781--809.

\bibitem{seeley2}
R.~Seeley, 
Pseudo-Differential Operators, 
In: Topics in Pseudo-Differential Operators,  C.I.M.E Summer Schools {\bf 47},  Springer,  Berlin (1969) 169--305.

\bibitem{strohmaier}
L.~Li and A.~Strohmaier,
The local counting function of operators of Dirac and Laplace type,
{\it J.~Geom.~Phys.} \textbf{104} (2016) 204--228.

%

\bibitem{rozenblum}
Y.~Miyanishi and G.~Rozenblum, Spectral properties of the Neumann--Poincar\'e operator in 3D elasticity, {\it Int. Math. Res. Not.} {\bf 2021} no.~11 (2021) 8715--8740.


%

\bibitem{NS04}
G.~Nenciu and V.~Sordoni,
Semiclassical limit for multistate Klein--Gordon systems: almost invariant subspaces and scattering theory,
{\it J. Math.  Phys.}  {\bf 45} (2004) 3676.


\bibitem{PST03}
G.~Panati,  H.~Spohn and S.~Teufel, 
Space-adiabatic perturbation theory,
{\it Adv.  Theor.  Math.  Phys. } {\bf 7} (2003) 145--204.

\bibitem{safarov}
Yu.~Safarov,
{\it Non-classical two-term spectral asymptotics for self-adjoint elliptic operators}.
 DSc~thesis, Leningrad Branch of the Steklov Mathematical Institute of the USSR Academy of Sciences (1989). In Russian.

%
%

\bibitem{sternin}
A.~Savin and B.~Sternin,
Pseudodifferential subspaces and their applications in elliptic theory,
in: {\it $C^*$-algebras and Elliptic Theory}, B.~Bojarski, A.~S.~Mishchenko, E.~V.~Troitsky and A.~Weber (Eds.), Trends in Mathematics, Birkh\"auser Basel (2006) 247--289. 

\bibitem{seeley}
R.~T.~Seeley, 
Complex powers of an elliptic operator, In: {\it Proc. Symp. Pure Math.} \textbf{10},
Amer. Math. Soc., Providence (RI), 1967, 288--307.

%
\bibitem{shubin}
M.~A.~Shubin,
{\it Pseudodifferential operators and spectral theory}.
Springer, 2001.
%
%
%


\bibitem{woj1}
K.~P.~Wojciechowski,
A note on the space of pseudodifferential projections with the same principal symbol,
{\it J. Operator Theory} \textbf{15} no.~2 (1986) 207--216. 

\bibitem{woj2}
K.~P.~Wojciechowski,
On the Calder\'on projections and spectral projections of the elliptic operators,
{\it J. Operator Theory} \textbf{20} no.~1 (1988) 107--115. 

\end{thebibliography}
\end{document}